\documentclass{aims}
\usepackage{amsmath,amsthm,amscd,amssymb,bm}
  \usepackage{paralist}
  \usepackage{graphics} 
  \usepackage{epsfig} 
\usepackage{graphicx}
\usepackage{enumerate}
\usepackage{epstopdf}
\usepackage{algpseudocode}
\usepackage{algorithm}
 \usepackage[colorlinks=true]{hyperref}
\hypersetup{urlcolor=blue, citecolor=red}

\def\currentvolume{X}
 \def\currentissue{X}
  \def\currentyear{20xx}
   \def\currentmonth{XX}
    \def\ppages{X--XX}
     \def\DOI{10.3934/ipi.xx.xx.xx}

\providecommand{\norm}[1]{\left\lVert#1\right\rVert}

\newcommand{\vb}{\mathbf{b}}

\newcommand{\vd}{\mathbf{d}}
\newcommand{\ve}{\mathbf{e}}

\newcommand{\vn}{\mathbf{n}}
\newcommand{\vr}{\mathbf{r}}

\newcommand{\vw}{\mathbf{w}}
\newcommand{\vx}{\mathbf{x}}
\newcommand{\vy}{\mathbf{y}}

\newcommand{\vo}{\mathbf{0}}

\newcommand{\cD}{\mathcal{D}}

\DeclareMathOperator*{\argmin}{argmin}
\DeclareMathOperator{\supp}{supp}
\newcommand{\st}{\mathrm{s.t.}}

\newtheorem{theorem}{Theorem}[section]

\newtheorem{lemma}[theorem]{Lemma}

\theoremstyle{definition}
\newtheorem{definition}[theorem]{Definition}

\newtheorem{assump}{Assumption}{}

\usepackage{xcolor}

\title[Stochastic Greedy Algorithms for MMV]{Stochastic Greedy Algorithms For Multiple Measurement Vectors}

\author[J. Qin, S. Li, D. Needell, A. Ma, R. Grotheer, C. Huang and N. Durgin]{}


\email{jing.qin@uky.edu}
\email{shuangli@mymail.mines.edu}
\email{deanna@math.ucla.edu}
\email{anna.ma@uci.edu}
\email{grotheerre@wofford.edu}
\email{chenxi.huang@yale.edu}
\email{njdurgin@gmail.com}

\thanks{$^*$ Corresponding author: Jing Qin}

\begin{document}
\maketitle
\centerline{\scshape Jing Qin$^*$}
\medskip
{\footnotesize
   \centerline{Department of Mathematics}
   \centerline{University of Kentucky, Lexington, KY 40506}
}
\medskip
\centerline{\scshape Shuang Li}
\medskip
{\footnotesize
   \centerline{Department of Mathematics}
   \centerline{University of California, Los Angeles, CA 90095}
}

\medskip
\centerline{\scshape Deanna Needell}
\medskip
{\footnotesize
   \centerline{Department of Mathematics}
   \centerline{University of California, Los Angeles, CA 90095}
}

\medskip
\centerline{\scshape Anna Ma}
\medskip
{\footnotesize
   \centerline{Department of Mathematics}
   \centerline{University of California, Irvine, CA 92697}
}

\medskip
\centerline{\scshape Rachel Grotheer}
\medskip
{\footnotesize
   \centerline{Center for Data, Mathematical, and Computational Sciences}
   \centerline{Goucher College, Baltimore, MD 21204}
}

\medskip
\centerline{\scshape Chenxi Huang}
\medskip
{\footnotesize
   \centerline{Center for Outcomes Research and Evaluation}
   \centerline{Yale University, New Haven, CT 06511}
}

\medskip
\centerline{\scshape Natalie Durgin}
\medskip
{\footnotesize
   \centerline{Spiceworks, Austin, TX 78746}
}

\bigskip

\begin{abstract}
Sparse representation of a single measurement vector (SMV) has been explored in a variety of compressive sensing applications. Recently, SMV models have been extended to solve multiple measurement vectors (MMV) problems, where the underlying signal is assumed to have joint sparse structures. To circumvent the NP-hardness of the $\ell_0$ minimization problem, many deterministic MMV algorithms solve the convex relaxed models with limited efficiency. In this paper, we develop stochastic greedy algorithms for solving the joint sparse MMV reconstruction problem. In particular, we propose the MMV Stochastic Iterative Hard Thresholding (MStoIHT) and MMV Stochastic Gradient Matching Pursuit (MStoGradMP) algorithms, and we also utilize the mini-batching technique to further improve their performance. Convergence analysis indicates that the proposed algorithms are able to converge faster than their SMV counterparts, i.e., concatenated StoIHT and StoGradMP, under certain conditions. Numerical experiments have illustrated the superior effectiveness of the proposed algorithms over their SMV counterparts.
\end{abstract}

\section{Introduction}\label{sec:intro}
Reconstruction of sparse signals from limited measurements has been studied extensively with a variety of applications in various imaging sciences, machine learning, computer vision and so on. The major problem is to reconstruct a signal which is sparse by itself or in some transformed domain from a small number of measurements (or observations). Let $\vx\in\mathbb{R}^n$ be the signal to be reconstructed. Then the sparse signal reconstruction problem can be formulated as an $\ell_0$ constrained minimization problem
\begin{equation}\label{eqn:SMVmodel}
\min_{\vx\in\mathbb{R}^n}F(\vx),\quad\mbox{s.t.}\quad \norm{\vx}_0\leq k,
\end{equation}
where the sparsity $\norm{\vx}_0$ counts the number of nonzero elements in $\vx$. Here $F(\vx)$ is a loss function measuring the discrepancy between the acquired measurements $\vy\in\mathbb{R}^{m}$ ($m\ll n$) and the measurements predicted by the estimated solution. In particular, if the measurements are linearly related to the underlying signal, i.e., there exists a sensing matrix $A\in\mathbb{R}^{m\times n}$ such that $\vy=A\vx+\vn$ where $\vn$ is the Gaussian noise, then the least squares loss function is widely used:
$F(\vx)=\frac12\norm{\vy-A\vx}_2^2$.
Then \eqref{eqn:SMVmodel} becomes a single measurement vector (SMV) sparse signal reconstruction problem \cite{cotter2005sparse}. The choice of $F$ depends on the generation mechanism of the measurements. Since the measurements are often generated in real time in many imaging techniques, it becomes significantly important in practice to reconstruct a collection of sparse signals, expressed as a signal matrix, from multiple measurement vectors (MMV). More precisely, the signal matrix $X\in\mathbb{R}^{n\times L}$ with $k$ ($k\leq n$) nonzero rows can be obtained by solving the following MMV model
\begin{equation}\label{eqn:MMVmodel}
\min_{X\in\mathbb{R}^{n\times L}}F(X),\quad\mbox{s.t.}\quad \norm{X}_{r,0}\leq k,
\end{equation}
where $\norm{X}_{r,0}$ stands for the row-sparsity of $X$ which counts nonzero rows in $X$. In the MMV literature, the column size $L$ of $X$ typically represents the number of snapshots of a signal. Note that it is possible that certain columns of $X$ have more zero components than zero rows of $X$. The MMV sparse reconstruction problem was first introduced in magnetoencephalography (MEG) imaging \cite{cotter2005sparse}, and has been extended to other applications \cite{he2008cg,bazerque2010distributed,majumdar2011joint,majumdar2012face,majumdar2013rank,davies2012rank,li2017atomic}.

Many SMV algorithms can be applied to solve MMV problems. The most straightforward way is to use SMV algorithms to reconstruct each signal vector sequentially or simultaneously via parallel computing, and then concatenate all resultant signals to form the estimated signal matrix. We call these types of algorithms, \emph{concatenated SMV} algorithms. On the other hand, the MMV problem can be converted to an SMV one by column-wise stacking the unknown signal matrix $X$ as a vector and introducing a block diagonal matrix as the new sensing matrix $A$. However, neither approach fully takes advantage of the joint sparse structure of the underlying signal matrix, and lack computational efficiency as well. In this paper, we develop MMV algorithms without concatenation of the SMV results or vectorization of the unknown signal matrix.

Since the $\ell_0$ term in \eqref{eqn:SMVmodel} and \eqref{eqn:MMVmodel} is non-convex and non-differentiable, many classical convex optimization algorithms fail to produce a satisfactory solution. To handle the NP-hardness of the problem, many convex relaxation methods and their MMV extensions have been developed, e.g., the $\ell_2$-regularized M-FOCUSS \cite{cotter2005sparse} and the $\ell_1$-regularized MMV extensions of the alternating direction method of multipliers \cite{lu2011fast,Jian2015split}. By exploiting the relationship between the measurements and the correct atoms, multiple signal classification (MUSIC) \cite{schmidt1986multiple,feng1996spectrum} and its improved variants \cite{kim2012compressive,lee2012subspace} have been developed. However, in the rank defective cases when the rank of the measurement matrix is much smaller than the desired row-sparsity level, the MUSIC type of methods will mostly fail to identify the correct atoms. The third category of algorithms for solving the $\ell_0$ constrained problem is the class of greedy algorithms that seek the sparsest solution by updating the support iteratively. SMV greedy algorithms include Orthogonal Matching Pursuit (OMP) \cite{pati1993orthogonal,tropp2004greed} and its RIP condition discussion \cite{zhang2011sparse}, simultaneous OMP (S-OMP) \cite{chen2006theoretical,tropp2006algorithms}, Compressive Sampling Matching Pursuit (CoSaMP) \cite{needell2009cosamp}, Regularized OMP (ROMP) \cite{needell2010signal}, Subspace-Pursuit (SP) \cite{dai2009subspace}, Iterative Hard-Thresholding (IHT) \cite{blumensath2009iterative}, Normalized IHT \cite{blumensath2010normalized}, Hard Thresholding Pursuit (HTP) \cite{foucart2011hard}, gradient hard thresholding \cite{yuan2014gradient}, and conjugate gradient IHT (CGIHT) \cite{blanchard2015cgiht}. It has been shown that CoSaMP and IHT are more efficient than the convex relaxation methods with strong recovery guarantees \cite{needell2009cosamp}. Greedy algorithms have also been extended to cosparse analysis to find the nearest cosparse subspace to a vector \cite{giryes2014greedy}. However, most of these algorithms work for compressive sensing applications where $F$ is a least squares loss function. Recently, the Gradient Matching Pursuit (GradMP) \cite{nguyenunified} has been proposed to extend CoSaMP to handle more general loss functions. To further improve efficiency and consider the non-convex objective function case, Stochastic IHT (StoIHT) and Stochastic GradMP (StoGradMP) have been proposed \cite{nguyen2014linear}. Nevertheless, concatenated extension of the aforementioned SMV greedy algorithms to the MMV versions will result in limited performance especially for large data sets. Using the row sparsity, MMV extensions of SMV greedy algorithms have been developed, e.g., M-OMP, M-SP, M-CoSaMP and M-HTP in \cite{blanchard2014greedy} and references therein. Based on countable and uncountable set of measurements, a class of models consider infinite measurements vectors (IMV) broader than MMV which have been proposed in \cite{mishali2008reduce}.

In this paper, we propose the MMV Stochastic IHT (MStoIHT) and the MMV Stochastic GradMP (MStoGradMP) algorithms for solving the general MMV joint sparse recovery problem \eqref{eqn:MMVmodel}. To accelerate convergence, the mini-batching technique is applied to the proposed algorithms. We also theoretically show that under certain conditions the proposed algorithms converge faster than their SMV concatenated counterparts, i.e., CStoIHT and CStoGradMP. Note that CStoIHT and CStoGradMP may fail to converge to a joint sparse solution without the columnwise separability of the objective function. The naming rule for the algorithms follows: initial C - concatenated, initial M - MMV, initial B - batched. A large variety of numerical experiments on joint sparse matrix recovery and video sequence recovery have demonstrated the superior performance of the proposed algorithms over their SMV counterparts in terms of running time and accuracy. Please refer to \cite{durgin2019fast} for our more recent work on MStoGradMP applied to solve the hyperspectral diffuse optical tomography imaging problem, which has shown the great potential of stochastic greedy algorithms on recovering large-scale data with joint sparsity.

\textbf{Organization.} The rest of the paper is organized as follows. Preliminary knowledge and notation clarifications are provided in Section \ref{sec:pre}. Section \ref{sec:algo} presents the concatenated SMV algorithms, and the proposed stochastic greedy algorithms, i.e., MStoIHT and MStoGradMP, in detail. Section \ref{sec:batch} discusses how to apply the mini-batching technique to accelerate the proposed algorithms. Convergence analysis is provided in Section \ref{sec:convergence}. By choosing the widely used least squares loss function as $F$, joint sparse signal recovery in distributed compressive sensing is discussed in Section \ref{sec:DCS}. Extensive numerical results are shown in Section \ref{sec:exp}. Finally, some concluding remarks are made in Section \ref{sec:con}.

\section{Preliminaries}\label{sec:pre}
To make the paper self-contained, we first introduce some useful notation and definitions, and then briefly describe the related algorithms, i.e., StoIHT and StoGradMP. Let $[m]=\{1,2,\ldots,m\}$ and $|\Omega|$ be the number of elements in the set $\Omega$. Consider a finite atom set $\mathcal{D}=\{\vd_1,\ldots,\vd_{N}\}$ (a.k.a. the dictionary) with each atom $\vd_i\in\mathbb{R}^n$.

\subsection{Vector and Matrix Sparsity}
Assume that a vector $\vx\in\mathbb{R}^n$ can be written as a linear combination of $\vd_i$'s, i.e., $\vx=\sum_{i=1}^{N}\alpha_i\vd_i=D\alpha$ with
$D=[\vd_1,\ldots,\vd_{N}]$ and $\alpha=(\alpha_1,\ldots,\alpha_N)^T$.
Then the support of $\vx$ with respect to $\alpha$ and $\mathcal{D}$ is defined by
\[
\supp_{\alpha,\mathcal{D}}(\vx)=\{i\in[N]:\alpha_i\neq0\}:=\supp(\alpha).
\]
The $\ell_0$-norm of $\vx$ with respect to $\mathcal{D}$ is defined as the minimal support size
\[
\norm{\vx}_{0,\mathcal{D}}=\min_{\alpha}\{|T|:\vx=\sum_{i\in T}\alpha_i\vd_i,T\subseteq[N]\}.
\]
Since absolute homogeneity does not hold in general, i.e., $\norm{\gamma\vx}_{0,\mathcal{D}}=|\gamma|\norm{\vx}_{0,\mathcal{D}}$ holds if and only if $|\gamma|=1$, this $\ell_0$-norm is not a norm. Here the smallest support $\supp_{\alpha,\mathcal{D}}(\vx)$ is called the support of $\vx$ with respect to $\mathcal{D}$, denoted by $\supp_{\mathcal{D}}(\vx)$. Thus
\[
|\supp_{\mathcal{D}}(\vx)|=\norm{\vx}_{0,\mathcal{D}}.
\]
Note that the support may not be unique if $\mathcal{D}$ is over-complete in that there could be multiple representations of $\vx$ with respect to the atom set $\mathcal{D}$ due to the linear dependence of the atoms in $\mathcal{D}$. In this paper, we assume that $\mathcal{D}$ is chosen so that the support of any arbitrary vector is unique, e.g., $\mathcal{D}$ is chosen as the standard basis in a Euclidean space or a learned basis. Furthermore, Donoho defines a general class of \emph{unconditional basis} in \cite{donoho1993unconditional}, e.g., the wavelet basis, which requires the unique representation and thereby is also applicable to our MMV algorithms. Without this uniqueness representation guarantee, we could either restrict the solution set or take a transformation of the unknown variable.

The vector $\vx$ is called \emph{$k$-sparse} with respect to $\mathcal{D}$ if $\norm{\vx}_{0,\mathcal{D}}\leq k$. For a subset $\Gamma\subseteq [N]$, we denote the set of atoms from $\mathcal{D}$ with indices restricted to $\Gamma$ by $\mathcal{D}_\Gamma$, and denote the subspace of $\mathbb{R}^n$ spanned by the atoms in $\mathcal{D}_\Gamma$ by $\mathcal{R}_{\mathcal{D}_\Gamma}$. Given a vector $\vw\in\mathbb{R}^n$, we use $\mathcal{P}_\Gamma \vw$ to denote the orthogonal projection of $\vw$ onto the subspace $\mathcal{R}(\mathcal{D}_\Gamma)$.

We extend vector sparsity and define the row sparsity for a matrix $X\in\mathbb{R}^{n\times L}$ as follows
\[
\norm{X}_{r,0,\mathcal{D}}=\min_{\Omega}\{|\Omega|:\Omega=\bigcup_{j=1}^L\supp_{\mathcal{D}}(X_{\cdot,j})\},
\]
where $X_{\cdot,j}$ is the $j$-th column of $X$. Here the minimal common support $\Omega$ is called the (row-wise) \emph{joint support} of $X$ with respect to $\mathcal{D}$, denoted by $\supp_{\mathcal{D}}^r(X)$, which satisfies
$|\supp_{\mathcal{D}}^r(X)|=\norm{X}_{r,0,\mathcal{D}}$.
The matrix $X$ is called $k$-row sparse with respect to $\mathcal{D}$ if all columns of $X$ share a joint support of size at most $k$ with respect to $\mathcal{D}$, i.e.,
$\norm{X}_{r,0,\mathcal{D}}\leq k$.

\subsection{Functions Defined on A Matrix Space}
Given a function $f:\mathbb{R}^{n\times L}\to \mathbb{R}$, the matrix derivative is defined by concatenating gradients \cite{magnus2010concept}
\begin{equation}\label{eqn:df}
\frac{\partial f}{\partial X}=\left[\frac{\partial f}{\partial X_{i,j}}\right]_{n\times L}
=\begin{bmatrix}\nabla_{X_{\cdot,1}}f&\cdots&\nabla_{X_{\cdot,L}}f\end{bmatrix},
\end{equation}
where $X_{i,j}$ is the $(i,j)$-th entry of $X$.
Notice that
\[
\norm{X}_F^2=\sum_{i=1}^n\norm{X_{i,\cdot}}_2^2=\sum_{j=1}^L\norm{X_{\cdot,j}}_2^2=\mathrm{Tr}(X^TX),
\]
where $X_{i,\cdot}$ is the $i$-th row vector of $X$, and $\mathrm{Tr}(\cdot)$ is the trace operator to add up all the diagonal entries of a matrix.
The inner product for any two matrices $U,V\in\mathbb{R}^{n\times L}$ is defined as
$\langle U,V\rangle=\mathrm{Tr}(U^TV)$.
Note that the equality
\[
\norm{U+V}_F^2=\norm{U}_F^2+\norm{V}_F^2+2\langle U,V\rangle
\]
and the Cauchy-Schwartz inequality
\[
\langle U,V\rangle \leq \norm{U}_F\norm{V}_F
\]
still hold. By generalizing the concepts in \cite{nguyen2014linear}, we define the $\mathcal{D}$-restricted strong convexity property and the strong smoothness property (a.k.a. the Lipschitz condition on the gradient) for the functions defined on a matrix space. It is worth noting the variants of these two concepts have been used to study the convergence of the projected
gradient descent algorithm \cite{agarwal2010fast}.
\begin{definition}\label{def:DRSC}
The function $f:\mathbb{R}^{n\times L}\to\mathbb{R}$ satisfies the $\mathcal{D}$-restricted strong convexity ($\mathcal{D}$-RSC) if there exists $\rho_{k}^->0$ such that
\begin{equation}\label{eqn:DRSC}
f(X')-f(X)-\Big\langle \frac{\partial f}{\partial X}(X),X'-X\Big\rangle\geq \frac{\rho_{k}^-}2\norm{X'-X}_F^2
\end{equation}
for $X',X\in\mathbb{R}^{n\times L}$ with $|\supp_{\mathcal{D}}^r(X)\cup\supp_{\mathcal{D}}^r(X')|\leq k$.
\end{definition}

\begin{definition}\label{def:DRSS}
The function $f:\mathbb{R}^{n\times L}\to \mathbb{R}$ satisfies the $\mathcal{D}$-restricted strong smoothness ($\mathcal{D}$-RSS) if there exists $\rho_{k}^+>0$ such that
\begin{equation}\label{eqn:DRSS}
\norm{\frac{\partial f}{\partial X}(X)-\frac{\partial f}{\partial X}(X')}_F\leq \rho_{k}^+\norm{X-X'}_F
\end{equation}
for $X',X\in\mathbb{R}^{n\times L}$ with $|\supp_{\mathcal{D}}^r(X)\cup\supp_{\mathcal{D}}^r(X')|\leq k$.
\end{definition}

\subsection{Related Work}
Many stochastic greedy algorithms based on matching pursuit and hard thresholding have been proposed to solve the SMV problem. In particular, GradMP generalizes CoSaMP to handle general convex objective functions \cite{nguyenunified}. Motivated by the success of stochastic convex optimization, StoIHT (see Algorithm~\ref{alg:StoIHT}) and StoGradMP (see Algorithm~\ref{alg:StoGradMP}) have been later proposed to solve the $\ell_0$ constrained SMV problem \cite{nguyen2014linear}
\begin{equation}\label{eqn:smvmodel}
\min_{\vx\in\mathbb{R}^n}\frac1M\sum_{i=1}^M\tilde{f}_i(\vx),\quad\mbox{subject to}\quad \norm{\vx}_{0,\mathcal{D}}\leq k,
\end{equation}
At each iteration of StoIHT, one component function $\tilde{f}_i:\mathbb{R}^n\to\mathbb{R}$ is first randomly selected with probability $p(i)$. Here the input discrete probability distribution $p(i)$'s satisfy $\sum_{i=1}^Mp(i)=1$, and $p(i)\geq0,\,i=1,\ldots,M$.
Next, in the ``Proxy'' step, gradient descent along the selected component is performed. Then the last two steps, i.e., ``Identify'' and ``Estimate'', essentially project the gradient descent result to its best $k$-sparse approximation.
Given $\vw=(w_1,\ldots,w_n)^T$ and $\eta\geq1$, the best $k$-sparse approximation operator acted on $\vw$ and $\eta$, denoted by $\mathrm{approx}_k(\vw,\eta)$, constructs an index set $\Gamma$ with $|\Gamma|=k$ such that
\begin{equation}\label{eqn:PGamma}
\norm{\mathcal{P}_\Gamma \vw-\vw}_2\leq \eta\norm{\vw-\vw_{(k)}}_2,\quad
\vw_{(k)}=\argmin_{\substack{\vy\in \mathcal{R}(\mathcal{D}_\Gamma)\\ |\Gamma|\leq k}}\norm{\vw-\vy}_2.
\end{equation}
Thus, if $\Gamma^*=\argmin_{\Gamma:|\Gamma|\leq k}\norm{\vw-\mathcal{P}_\Gamma\vw}_2$, we have $\norm{\vw-\mathcal{P}_\Gamma \vw}_2\leq \eta\norm{\vw-\mathcal{P}_{\Gamma^*}\vw}_2$, and $\vw_{(k)}$ defined in \eqref{eqn:PGamma} becomes the best $k$-sparse approximation of $\vw$ in the subspace $\mathcal{R}(\mathcal{D}_{\Gamma^*})$. In particular, if $\eta\geq1$ and $\mathcal{D}=\{\ve_i:\,i=1,2,\ldots,n\}$ with $\ve_i=[0,\ldots,\underset{(i)}{1},\ldots,0]^T$, then $\mathrm{approx}_k(\vw,\eta)$ returns the index set of the first $k$ largest entries of $\vw$ in absolute value, i.e.,
\[
\mathrm{approx}_k(\vw,1)=\{i_1,\ldots,i_k:|w_{i_1}|\geq\ldots\geq|w_{i_k}|\geq\ldots\geq |w_{i_n}|\}:=\widehat{\Gamma}.
\]
Then the projection $\mathcal{P}_\Gamma \vw$ reads as in componentwise form
\[
\big(\mathcal{P}_{\widehat{\Gamma}}(\vw)\big)_j=\left\{\begin{aligned}
&w_j&&\mbox{if }j\in\widehat{\Gamma},\\
&0&&\mbox{if }j\notin\widehat{\Gamma}.
\end{aligned}\right.
\]
There are two widely used stopping criteria: $\frac{\big\|\vx^{t+1}-\vx^t\big\|_2}{\norm{\vx^t}_2}<\varepsilon$, and
$\frac1M\sum_{i=1}^M\tilde{f}_i(\vx^t)<\varepsilon$,
where $\varepsilon>0$ is a small tolerance. It is well known that the first stopping criteria is more robust in practice \cite{burden2004numerical}.
\begin{algorithm}
\caption{Stochastic Iterative Hard Thresholding (StoIHT)}\label{alg:StoIHT}
\begin{algorithmic}
\State\textbf{Input:} $k,\gamma,\eta,p(i),\varepsilon$.
\State\textbf{Output:} $\widehat{\vx}=\vx^t$.
\State\textbf{Initialize:} $\vx^0=\mathbf{0}$.
\For {$t=0,1,2,\ldots,T$}
\State Randomly select an index $i_t\in\{1,2,\ldots,M\}$ with probability $p(i_t)$
\State Proxy: $\vb^t=\vx^t-\frac{\gamma}{Mp(i_t)}\nabla \tilde{f}_{i_t}(\vx^t)$
\State Identify: $\Gamma^t=\mathrm{approx}_k(\vb^t,\eta)$.
\State Estimate: $\vx^{t+1}=\mathcal{P}_{\Gamma^t}(\vb^t)$.
\State If the stopping criteria are met, exit.
\EndFor
\end{algorithmic}
\end{algorithm}

Different from StoIHT, StoGradMP involves gradient matching, i.e., seeking the best $k$-sparse approximation of the gradient rather than the solution. At the solution estimation step, the original problem is restricted to the components from the estimated support. It has been empirically shown that StoGradMP converges faster than StoIHT due to the more accurate estimation of the support \cite{nguyen2014linear}. But StoGradMP requires that the sparsity level $k$ is no more than $n/2$.
\begin{algorithm}
\caption{Stochastic Gradient Matching Pursuit (StoGradMP)}\label{alg:StoGradMP}
\begin{algorithmic}
\State\textbf{Input:} $k,\eta_1,\eta_2,p(i),\varepsilon$.
\State\textbf{Output:} $\widehat{\vx}=\vx^{t}$.
\State\textbf{Initialize:} $\vx^0=\mathbf{0}$, $\Lambda=\emptyset$.
\For {$t=0,1,2,\ldots,T$}
\State Randomly select an index $i_t\in\{1,2,\ldots,M\}$ with probability $p(i_t)$
\State Calculate the gradient $\vr^t=\nabla \tilde{f}_{i_t}(\vx^t)$
\State $\Gamma=\mathrm{approx}_{2k}(\vr^t,\eta_1)$
\State $\widehat{\Gamma}=\Gamma\cup\Lambda$
\State $\vb^t=\argmin_{\vx} \frac1M\sum_{i=1}^M\tilde{f}_i(\vx),\quad \vx\in\mathcal{R}(\mathcal{D}_{\widehat{\Gamma}})$
\State $\Lambda=\mathrm{approx}_k(\vb^t,\eta_2)$
\State $\vx^{t+1}=\mathcal{P}_{\Lambda}(\vb^t)$
\State If the stopping criteria are met, exit.
\EndFor
\end{algorithmic}
\end{algorithm}

\section{Proposed Stochastic Greedy Algorithms}\label{sec:algo}
In this section, we present concatenated SMV algorithms, and develop stochastic greedy algorithms for MMV problems based on StoIHT and StoGradMP. Suppose that there are $M$ differentiable and convex functions $f_i:\mathbb{R}^{n\times L}\to\mathbb{R}$ that satisfy the $\mathcal{D}$-restricted strong smoothness property (see Definition~\ref{def:DRSS}), and their mean
\begin{equation}\label{eqn:F}
F(X)=\frac1M\sum_{i=1}^Mf_i(X)
\end{equation}
satisfies the $\mathcal{D}$-restricted strong convexity property (see Definition~\ref{def:DRSC}). These assumptions will be used extensively throughout the entire paper. Consider the row-sparsity constrained MMV problem
\begin{equation}\label{eqn:model}
\min_{X\in\mathbb{R}^{n\times L}}\frac1M\sum_{i=1}^Mf_i(X),\quad\mbox{subject to}\quad \norm{X}_{r,0,\mathcal{D}}\leq k.
\end{equation}
By vectorizing $X$, i.e., rewriting $X$ as a vector $\vx\in\mathbb{R}^{nL}$ by columnwise stacking, we can \emph{relax} \eqref{eqn:model} to a sparsity constrained SMV problem of the form \eqref{eqn:smvmodel} where the sparsity level $k$ is replaced by $kL$. Since $\norm{\vx}_{0,\mathcal{D}}\leq kL$ does not necessarily guarantee $\norm{X}_{r,0,\mathcal{D}}\leq k$, the solution to the relaxed problem may not be the same as the vectorization of the solution to \eqref{eqn:model}. On the other hand, the iterative stochastic algorithms such as StoIHT and StoGradMP, can be developed to the \emph{concatenated} versions for solving \eqref{eqn:model} under the following assumption on the objective function $f_i$'s.

\begin{definition}\label{def:colsep}
Each component $f_i$ of the objective function $F$ in \eqref{eqn:model} is {columnwise separable}, in the sense that a collection of functions $g_{i,j}:\mathbb{R}^n\to\mathbb{R}$ exist with
\begin{equation}\label{eqn:colsep}
f_i(X)=\sum_{j=1}^Lg_{i,j}(X_{\cdot,j}),\quad i=1,\ldots,M.
\end{equation}
In this case, the objective function can be rewritten as
\begin{equation}\label{eqn:Fj}
F(X)=\sum_{j=1}^L\left(\frac1M\sum_{i=1}^Mg_{i,j}(X_{\cdot,j})\right):=\sum_{j=1}^L\widehat{F_j}(X_{\cdot,j}),
\end{equation}
which implies that minimizing $F$ with respect to $X$ can be reduced to minimizing $\widehat{F_j}$ for $j=1,\ldots,L$ individually.
\end{definition}

Under this assumption, the concatenated algorithms, i.e., CStoIHT in Algorithm~\ref{alg:CStoIHT} and CStoGradMP in Algorithm~\ref{alg:CStoGradMP}, can be applied to solve \eqref{eqn:model}, which essentially reconstruct each column of $X$ by solving the SMV problem \eqref{eqn:smvmodel}. Notice that the outer loops of CStoIHT and CStoGradMP can be executed in a parallel manner on a multi-core computer, when the order of the inner loop and the outer loop in Algorithm~\ref{alg:CStoIHT} can be swapped. Note that although a joint support is not enforced at each iteration, the paralleled version of CStoIHT and CStoGradMP can still yield joint sparse results with guaranteed convergence (see Theorem~\ref{thm:SMVStoIHT} and Theorem~\ref{thm:SMVStoGradMP}). However, if the sparsity level $k$ is very large, then the support sets of $X_{\cdot,j}$'s are prone to overlap less initially which results in the less accurate estimation of the joint support and larger errors in the initial iterates. In addition, for some nonlinear function $f_i(X)$ which can not be separated as a sum of functions for columns of $X$, it will be challenging to find an appropriate objective function for the corresponding SMV problem.

\begin{algorithm}
\caption{Concatenated Stochastic Iterative Hard Thresholding (CStoIHT)}\label{alg:CStoIHT}
\begin{algorithmic}
\State\textbf{Input:} $k,\gamma,\eta,p(i),\varepsilon$.
\State\textbf{Output:} $\widehat{X}=X^t$.
\State\textbf{Initialize:} $X^0=\mathbf{0}\in\mathbb{R}^{n\times L}$.
\For {$j=1,\ldots,L$}
\For {$t=0,1,2,\ldots,T$}
\State Randomly select an index $i_t\in\{1,2,\ldots,M\}$ with probability $p(i_t)$
\State Proxy: $\vb^t=X_{\cdot,j}^t-\frac{\gamma}{Mp(i_t)}\nabla g_{i_t,j}(X_{\cdot,j}^t)$
\State Identify: $\Gamma^t=\mathrm{approx}_k(\vb^t,\eta)$.
\State Estimate: $X_{\cdot,j}^{t+1}=\mathcal{P}_{\Gamma^t}(\vb^t)$.
\State If the stopping criteria are met, exit.
\EndFor
\EndFor
\end{algorithmic}
\end{algorithm}

\begin{algorithm}
\caption{Concatenated Stochastic Gradient Matching Pursuit (CStoGradMP)}\label{alg:CStoGradMP}
\begin{algorithmic}
\State\textbf{Input:} $k,\eta_1,\eta_2,p(i),\varepsilon$.
\State\textbf{Output:} $\widehat{X}=X^{t}$.
\State\textbf{Initialize:} $X^0=\mathbf{0}\in\mathbb{R}^{n\times L}$, $\Lambda=\emptyset$.
\For {$j=1,\ldots,L$}
\For {$t=0,1,2,\ldots,T$}
\State Randomly select an index $i_t\in\{1,2,\ldots,M\}$ with probability $p(i_t)$
\State Calculate the gradient $\vr^t=\nabla g_{i_t,j}(X_{\cdot,j}^t)$
\State $\Gamma=\mathrm{approx}_{2k}(\vr^t,\eta_1)$
\State $\widehat{\Gamma}=\Gamma\cup\Lambda$
\State $\vb^t=\argmin_{\vx} \frac1M\sum_{i=1}^Mg_{i,j}(\vx),\quad \vx\in\mathcal{R}(\mathcal{D}_{\widehat{\Gamma}})$
\State $\Lambda=\mathrm{approx}_k(\vb^t,\eta_2)$
\State $X_{\cdot,j}^{t+1}=\mathcal{P}_{\Lambda}(\vb^t)$
\State If the stopping criteria are met, exit.
\EndFor
\EndFor
\end{algorithmic}
\end{algorithm}

To circumvent the aforementioned issues, we first propose the MMV Stochastic Iterative Hard Thresholding algorithm (MStoIHT) detailed in Algorithm \ref{alg:MStoIHT}. Compared to StoIHT, MStoIHT replaces the gradient by the matrix derivative \eqref{eqn:df}. The second significant difference lies in the ``Identify" and ``Estimate" steps, especially the operator $\mathrm{approx}_k(\cdot,\cdot)$. Now we extend the operator $\mathrm{approx}_k(\cdot,\cdot)$ from sparse vectors to row-sparse matrices. Given $X\in\mathbb{R}^{n\times L}$ and $\eta\geq1$, the best $k$-row sparse approximation operator acted on $X$ and $\eta$, denoted by $\mathrm{approx}_k^r(X,\eta)$, constructs a \emph{row} index set $\Gamma$ such that
\[
\norm{\mathcal{P}_{\Gamma}X_{\cdot,j}-X_{\cdot,j}}_2\leq \eta\norm{X_{\cdot,j}-(X_{\cdot,j})_k}_2,\quad j=1,\ldots,L,
\]
where $(X_{\cdot,j})_k$ is the best $k$-sparse approximation of the column vector $X_{\cdot,j}$ with respect to $\mathcal{D}$.
Note that we can choose $\Gamma$ as the union of all $\Gamma_j$'s if $\Gamma_j$ is the index set satisfying the inequality with a specified index $j$.
In particular, if $\mathcal{D}=\{\ve_i:i=1,\ldots,n\}$ and $\eta=1$, $\mathrm{approx}_{k}^r(X,1)$ returns the row index set of the first $k$ largest $\ell_2$ row norms in $X$, i.e.,
\[
\mathrm{approx}_k^r(X,1)=\{i_1,\ldots,i_k:\|X_{i_1,\cdot}\|_2\geq\ldots\geq \|X_{i_n,\cdot}\|_2\}:=\widetilde{\Gamma}.
\]
By abusing the notation, we define $\mathcal{P}_{\widetilde{\Gamma}}(X)$ to be the projection of $X$ onto the subspace of all row-sparse matrices with row indices restricted to $\widetilde{\Gamma}$. Therefore, we have
$
\mathcal{P}_{\widetilde{\Gamma}}X
=\begin{bmatrix}\mathcal{P}_{\widetilde{\Gamma}}X_{\cdot,1}&
\ldots&\mathcal{P}_{\widetilde{\Gamma}}X_{\cdot,L}\end{bmatrix}.
$
Due to the common support $\widetilde{\Gamma}$, the projection $\mathcal{P}_{\widetilde{\Gamma}}(X)$ can also be written as in row-wise form
\begin{equation}\label{eqn:proj}
\big(\mathcal{P}_{\widetilde{\Gamma}}(X)\big)_{j,\cdot}=\left\{
\begin{aligned}
&X_{j,\cdot}&&\mbox{if }j\in\widetilde{\Gamma},\\
&\vo&&\mbox{if }j\notin\widetilde{\Gamma}.
\end{aligned}
\right.
\end{equation}
Here $\mathcal{P}_{\widetilde{\Gamma}}(X)$ returns a $k$-row sparse matrix, whose nonzero rows correspond to the $k$ rows of $X$ with largest $\ell_2$ row norms.
\begin{algorithm}
\caption{MMV Stochastic Iterative Hard Thresholding (MStoIHT)}\label{alg:MStoIHT}
\begin{algorithmic}
\State\textbf{Input:} $k,\gamma,\eta,p(i),\varepsilon$.
\State\textbf{Output:} $\widehat{X}=X^t$.
\State\textbf{Initialize:} $X^0=\mathbf{0}$.
\For {$t=0,1,2,\ldots,T$}
\State Randomly select an index $i_t\in\{1,2,\ldots,M\}$ with probability $p(i_t)$
\State Proxy: $B^t=X^t-\frac{\gamma}{Mp(i_t)}\frac{\partial f_{i_t}(X^t)}{\partial X}$
\State Identify: $\Gamma^t=\mathrm{approx}_k^r(B^t,\eta)$.
\State Estimate: $X^{t+1}=\mathcal{P}_{\Gamma^t}(B^t)$.
\State If the stopping criteria are met, exit.
\EndFor
\end{algorithmic}
\end{algorithm}

Next, we propose the MMV Stochastic Gradient Matching Pursuit (MStoGradMP) detailed in Algorithm \ref{alg:MStoGradMP}, where the two gradient matching steps involve the operator $\mathrm{approx}_{k}^r(\cdot,\cdot)$. The stopping criteria in all proposed algorithms can be set as the same as those in Algorithm~\ref{alg:StoIHT} and Algorithm~\ref{alg:StoGradMP}. Here we choose the stopping criterion: ${\big\|X^{t+1}-X^t\big\|_F}/{\norm{X^t}_F}<\varepsilon$, where $\varepsilon>0$ is a small tolerance.
\begin{algorithm}
\caption{MMV Stochastic Gradient Matching Pursuit (MStoGradMP)}\label{alg:MStoGradMP}
\begin{algorithmic}
\State\textbf{Input:} $k,\eta_1,\eta_2,p(i),\varepsilon$.
\State\textbf{Output:} $\widehat{X}=X^{t}$.
\State\textbf{Initialize:} $X^0=\mathbf{0}$, $\Lambda=\emptyset$.
\For {$t=0,1,2,\ldots,T$}
\State Randomly select an index $i_t\in\{1,2,\ldots,M\}$ with probability $p(i_t)$
\State Calculate the generalized gradient $R^t=\frac{\partial f_{i_t}(X^t)}{\partial X}$
\State $\Gamma=\mathrm{approx}_{2k}^r(R^t,\eta_1)$
\State $\widehat{\Gamma}=\Gamma\cup\Lambda$
\State $B^t=\argmin_X F(X),\quad X\in\mathcal{R}(\mathcal{D}_{\widehat{\Gamma}})$
\State $\Lambda=\mathrm{approx}_k^r(B^t,\eta_2)$
\State $X^{t+1}=\mathcal{P}_{\Lambda}(B^t)$
\State If the stopping criteria are met, exit.
\EndFor
\end{algorithmic}
\end{algorithm}

\section{Batched Acceleration}\label{sec:batch}
To accelerate computations and improve performance, we apply the mini-batching technique to obtain batched variants of Algorithms \ref{alg:MStoIHT} and \ref{alg:MStoGradMP}. We first partition the index set $\{1,2,\ldots,M\}$ into a collection of equal-sized batches $\tau_1,\ldots,\tau_d$ where the batch size $|\tau_i|=b$ for all $i=1,2,\ldots,\lceil M/b\rceil:=d$. For simplicity, we assume that $M/b$ is an integer. Similar to the approach in \cite{needell2016batched}, we reformulate \eqref{eqn:F} as
\begin{equation}\label{eqn:bmodel}
F(X)=\frac1{d}\sum_{i=1}^d\left(\frac1b\sum_{j\in\tau_i}f_j(X)\right):=\frac1d\sum_{i=1}^d\widehat{f}_i(X).
\end{equation}
That is, $\widehat{f}_{i}$ is the average of the $i$-th batch of component functions $\{f_j\}_{j\in\tau_i}$. Based on this new formulation, we get the batched version of Algorithm \ref{alg:MStoIHT}, which is termed as BMStoIHT, described in Algorithm \ref{alg:BMStoIHT}. Here the input probability $p(i)$ satisfies
\[
\frac1{d}\sum_{i=1}^dp(i)=1\quad\mbox{and}\quad p(i)\geq 0,\,i=1,\ldots,d.
\]
Likewise, we get a batched version of MStoGradMP, termed as BMStoGradMP, which is detailed in Algorithm \ref{alg:BMStoGradMP}.
Note that $\hat{f}_{\tau_t}$ in Algorithms \ref{alg:BMStoIHT} and \ref{alg:BMStoGradMP} is the $\tau_t$-th component function defined in \eqref{eqn:bmodel}, and $\frac{\partial \hat{f}_{\tau_t}(X^t)}{\partial X}$ is the derivative of $\hat{f}_{\tau_t}$ defined in \eqref{eqn:df}.
It is empirically shown in Section \ref{sec:exp} that the increase of the batch size greatly speeds up the convergence of both algorithms, which is more obvious in BMStoIHT. As a by-product, mini-batching can also improve the recovery accuracy based on our experiments. However, there is a trade-off between the batch size and the performance improvement as mentioned in \cite{needell2016batched}.

\begin{algorithm}
\caption{Batched MMV Stochastic Iterative Hard Thresholding (BMStoIHT)}\label{alg:BMStoIHT}
\begin{algorithmic}
\State\textbf{Input:} $k,\gamma,\eta$, $b$ and $p(i)$.
\State\textbf{Output:} $\widehat{X}=X^t$.
\State\textbf{Initialize:} $X^0=\mathbf{0}$.
\For {$t=0,1,2,\ldots,T$}
\State Randomly select a batch index $\tau_t\subseteq\{1,2,\ldots,d\}$ of size $b$ with probability $p(\tau_t)$
\State Proxy: $B^t=X^t-\frac{\gamma}{dp(\tau_t)}\frac{\partial \widehat{f}_{\tau_t}(X^t)}{\partial X}$
\State Identify $\Gamma^t=\mathrm{approx}_k^r(B^t,\eta)$.
\State Estimate $X^{t+1}=\mathcal{P}_{\Gamma^t}(B^t)$.
\State If the stopping criteria are met, exit.
\EndFor
\end{algorithmic}
\end{algorithm}

\begin{algorithm}
\caption{Batched MMV Stochastic Gradient Matching Pursuit (BMStoGradMP)}\label{alg:BMStoGradMP}
\begin{algorithmic}
\State\textbf{Input:} $k,\eta_1,\eta_2$, $b$ and $p(i)$.
\State\textbf{Output:} $\widehat{X}=X^{t}$.
\State\textbf{Initialize:} $X^0=\mathbf{0}$, $\Lambda=\emptyset$.
\For {$t=0,1,2,\ldots,T$}
\State Randomly select a batch index $\tau_t\subseteq\{1,2,\ldots,d\}$ of size $b$ with probability $p(\tau_t)$
\State Calculate the generalized gradient $R^t=\frac{\partial \widehat{f}_{\tau_t}(X^t)}{\partial X}$
\State $\Gamma=\mathrm{approx}_{2k}^r(R^t,\eta_1)$
\State $\widehat{\Gamma}=\Gamma\cup\Lambda$
\State $B^t=\argmin_X F(X),\quad X\in\mathrm{span}(D_{\widehat{\Gamma}})$
\State $\Lambda=\mathrm{approx}_k^r(B^t,\eta_2)$
\State $X^{t+1}=\mathcal{P}_{\Lambda}(B^t)$
\State If the stopping criteria are met, exit.
\EndFor
\end{algorithmic}
\end{algorithm}

\section{Convergence Analysis}\label{sec:convergence}
In this section, we provide the convergence guarantees for the proposed MStoIHT and MStoGradMP, together with their SMV counterparts, i.e., CStoIHT and CStoGradMP. To simplify the discussion, the result at the $t$-th iteration of CStoIHT/CStoGradMP refers to the result obtained after $t$ inner iterations and $L$ outer iterations of Algorithm~\ref{alg:CStoIHT}/Algorithm~\ref{alg:CStoGradMP}, or equivalently the maximum number of inner iterations is set as $t$. Furthermore, all convergence results can be extended to their batched versions, i.e., BStoIHT and BMStoGradMP. Comparison of contraction coefficients shows that the proposed algorithms have faster convergence under the $\mathcal{D}$-RSC, $\mathcal{D}$-RSS and columnwise separability of the objective function (see Definition~\ref{def:colsep}). Similar to Section~\ref{sec:algo}, we consider the model \eqref{eqn:model} under the following assumptions.

\begin{assump} The objective function $F(X)=\frac1M\sum_{i=1}^Mf_i(X):\mathbb{R}^{n\times L}\to\mathbb{R}$ satisfies
\begin{enumerate}[(a)]
\item $f_i$ is columnwise separable with differentiable and convex functions $g_{i,j}$ that satisfy \eqref{eqn:colsep} and $\mathcal{D}$-RSS in vector form, i.e., there exists $\rho_{k}^+(i,j)>0$ such that
    \begin{equation}\label{eqn:rho_g}
    \norm{\nabla g_{i,j}(\vw)-\nabla g_{i,j}(\widehat{\vw})}_2\leq \rho_{k}^+(i,j)\norm{\vw-\widehat{\vw}}_2
    \end{equation}
    for all vectors $\vw,\widehat{\vw}\in\mathbb{R}^n$ with $|\supp_{\mathcal{D}}(\vw)\cup\supp_{\mathcal{D}}(\widehat{\vw})|\leq k$.

\item $\widehat{F_j}=\sum_ig_{i,j}$ satisfies $\mathcal{D}$-RSC in vector form, i.e., there exists $\rho_{k,j}^->0$ such that
    \begin{equation}\label{eqn:rho_Fj}
    \widehat{F_j}(\widehat{\vw})-\widehat{F_j}(\vw)-\langle \nabla F_j(\vw),\widehat{\vw}-\vw\rangle \geq \frac{\rho_{k,j}^-}2\norm{\widehat{\vw}-\vw}_2^2
    \end{equation}
    for all vectors $\vw,\widehat{\vw}\in\mathbb{R}^n$ with $|\supp_{\mathcal{D}}(\vw)\cup\supp_{\mathcal{D}}(\widehat{\vw})|\leq k$.
\end{enumerate}
\end{assump}

Similar to \cite{nguyen2014linear}, the $\mathcal{D}$-RSS property (a) and $\mathcal{D}$-RSC property (b) in this assumption can guarantee the uniqueness of the solution to \eqref{eqn:model}.

\begin{lemma}\label{lem:rss}
If Assumption (a) is satisfied, then $f_i$ satisfies $\mathcal{D}$-RSS with constant $\mu_k^+(i)=\max\limits_{1\leq j\leq L}\rho_k^+(i,j)$ for $i=1,\ldots,M$.
\end{lemma}

\begin{lemma}\label{lem:rsc}
If Assumption (b) is satisfied, then $F(X)$ satisfies $\mathcal{D}$-RSC with constant $\mu_k^-=\min\limits_{1\leq j\leq L}\rho_{k,j}^-$.
\end{lemma}

The proofs of Lemma~\ref{lem:rss} and Lemma~\ref{lem:rsc} are straightforward, which can be found in the Appendix. Throughout the paper, we denote
\begin{equation}\label{eqn:alpha_k_j}
\begin{aligned}
\alpha_{k,j}&=\max\limits_{1\leq i\leq M}\tfrac{\rho_k^+(i,j)}{Mp(i)},\\
\rho_{k,j}^+&=\max\limits_{1\leq i\leq M}\rho_k^+(i,j),\\
\bar{\rho}_{k,j}^+&=\tfrac1M\sum\limits_{i=1}^M\rho_k^+(i,j),
\end{aligned}
\end{equation}
where $\rho_{k}^+(i,j)$ is the $\mathcal{D}$-RSS constant for $g_{i,j}$ defined in Assumption (a).

\subsection{MStoIHT}
By replacing the $\ell_2$-norm and vector inner product with the Frobenius norm and the matrix inner product respectively and using the properties of inner product in Section~\ref{sec:pre}, we get similar convergence results for MStoIHT in \cite{nguyen2014linear} in the following theorem. Here we skip the proof which is a straightforward matrix extension of \cite[Appendix B]{nguyen2014linear}.

\begin{theorem}[MStoIHT]\label{thm1}
Let $X^*$ be a feasible solution of \eqref{eqn:model} and $X^0$ be the initial solution. Assume that $F(X)=\frac1M\sum_{i=1}^Mf_i(X)$ satisfies the $\cD$-RSC with constant $\rho_k^-$ in \eqref{eqn:DRSC} and $f_i$ satisfies $\cD$-RSS with constant $\rho_k^+(i)$ in \eqref{eqn:DRSS}. Then at the $(t+1)$-th iteration, the expectation of the recovery error of Algorithm \ref{alg:MStoIHT} is bounded by
\begin{equation}\label{eqn:alg1E}
E\norm{X^{t+1}-X^*}_F\leq \kappa^{t+1}\norm{X^0-X^*}_F+\frac{\sigma_{X^*}}{1-\kappa},
\end{equation}
where the contraction coefficient $\kappa$ and the tolerance parameter $\sigma_{X^*}$ are
\begin{equation}\label{eqn:kappa_MStoIHT}
\begin{aligned}
\kappa&=2\sqrt{1-\gamma\left(2-\gamma{\alpha}_{3k}\right)\rho_{3k}^-}+\sqrt{(\eta^2-1)\left(1+\gamma^2{\alpha}_{3k}\sum_{i=1}^M\rho_{3k}^+(i)-2\gamma\rho_{3k}^-\right)},\\
\sigma_{X^*}&=\frac{\gamma}{\min\limits_{1\leq i\leq M}Mp(i)}\left(2E_i\max_{|\Omega|\leq 3k}\norm{\mathcal{P}_\Omega\frac{\partial f_i}{\partial X}(X^*)}_F+\sqrt{\eta^2-1}E_i\norm{\frac{\partial f_i}{\partial X}(X^*)}_F\right).
\end{aligned}
\end{equation}
Here $\mathcal{P}_\Omega$ is defined in \eqref{eqn:proj} and
\begin{equation}\label{eqn:rho+}
{\alpha}_{k}=\max_{1\leq i\leq M}\frac{\rho_k^+(i)}{Mp(i)},\,
\rho_k^+=\max_{1\leq i\leq M}\rho_k^+(i),\,\bar{\rho}_k^+=\frac1M\sum_{i=1}^M\rho_k^+(i).
\end{equation}
Thus Algorithm \ref{alg:BMStoIHT} converges linearly if $\kappa<1$. In particular, if $\eta=\gamma=1$, then
\begin{equation}\label{eqn:kappa_MStoIHT_CS}
\kappa=2\sqrt{1-2\rho_{3k}^-+\alpha_{3k}\rho_{3k}^-}.
\end{equation}
\end{theorem}

\noindent\textbf{Remark.} From \eqref{eqn:kappa_MStoIHT}, one can see that the tolerance parameter $\sigma_{X^*}$ increases as the step size $\gamma$ grows which implies that the step size cannot be too large. The contraction coefficient $\kappa$ which controls the convergence speed depends on the $\mathcal{D}$-RSC and $\mathcal{D}$-RSS constants. In addition, in the simplest case when $\eta=\gamma=1$ and $p(i)$ is a uniform distribution, one can see that $\kappa$ decreases if either the $\mathcal{D}$-RSC constant $\rho_{3k}^-$ increases or the maximum of all $\mathcal{D}$-RSS constants $\rho_{3k}^+(i)$ decreases according to \eqref{eqn:kappa_MStoIHT_CS}. That implies that the proposed algorithms will converge fast when each component function of the objective has either a small $\mathcal{D}$-RSS constant or a large $\mathcal{D}$-RSC constant in this case.

\subsection{CStoIHT}
To solve the problem \eqref{eqn:model}, CStoIHT uses StoIHT to reconstruct each column of $X$ separately and then concatenate all column vectors to form a matrix. To analyze the convergence of CStoIHT, we first derive an upper bound for $E\norm{X_{\cdot,j}-X_{\cdot,j}^*}_2^2$ following the proof outline in \cite[Appendix B]{nguyen2014linear}. Note that in our proof we look for the contraction coefficient of $E\norm{X_{\cdot,j}-X_{\cdot,j}^*}_2^2$ rather than $E\norm{X_{\cdot,j}-X_{\cdot,j}^*}_2$ as in \cite{nguyen2014linear}.

\begin{lemma}\label{lem:E2StoIHT}
Let $X^*$ be a feasible solution of \eqref{eqn:model} and $X^0$ be the initial solution. Under the Assumptions, there exist $\kappa_j,\sigma_j>0$ such that the expectation of the recovery error squared at the $t$-th iteration of Algorithm~\ref{alg:CStoIHT} for estimating the $j$-th column of $X^*$ is bounded by
\begin{equation}\label{eqn:E2StoIHT}
E_{I_t}\norm{X_{\cdot,j}^{t+1}-X_{\cdot,j}^*}_2^2\leq \kappa_j^{t+1}\norm{X_{\cdot,j}^0-X_{\cdot,j}^*}_2^2+\frac{\sigma_j}{1-\kappa_j},
\end{equation}
where $X_{\cdot,j}^t$ is the approximation of $X_{\cdot,j}^*$ at the $t$-th iteration of StoIHT with the initial guess $X_{\cdot,j}^0$, i.e., the result at the $t$-th inner iteration and $j$-th outer iteration of Algorithm~\ref{alg:CStoIHT} with the initial guess $X^0$. Here $I_t$ is the set of all indices $i_1,\ldots,i_t$ randomly selected at or before the $t$-th step of the algorithm.
\end{lemma}

\begin{proof}
Due to the separable form of $f_i$ in \eqref{eqn:colsep}, we consider $L$ problems of the form
\begin{equation}\label{eqn:SMVobj}
\min_{\vw}\frac1M\sum_{i=1}^Mg_{i,j}(\vw),\quad \norm{\vw}_{0,\mathcal{D}}\leq k,\quad j=1,\ldots,L,
\end{equation}
where $g_{i,j}$ are given in \eqref{eqn:colsep}. This relaxation is valid since $X^*_{\cdot,j}$ is also a feasible solution of \eqref{eqn:SMVobj} if $X^*$ is a feasible solution of \eqref{eqn:model}.
Let $\vw^t=X_{\cdot,j}^t$, $\vw^*=X_{\cdot,j}^*$, $x=\norm{\vw^{t+1}-\vw^*}_2$,
\[\begin{aligned}
u&=\norm{\vw^t-\vw^*-\frac{\gamma}{Mp(i_t)}\mathcal{P}_\Omega(\nabla g_{i_t,j}(\vw^t)-\nabla g_{i_t,j}(\vw^*))}_2+\norm{\frac{\gamma}{Mp(i_t)}\mathcal{P}_\Omega\nabla g_{i_t,j}(\vw^*)}_2\\
&:=u_1+u_2,
\end{aligned}\]
and
\[
v=(\eta^2-1)\norm{\vw^t-\vw^*-\frac{\gamma}{Mp(i_t)}\nabla g_{i_t,j}(\vw^t)}_2^2.
\]
Let
\[\begin{aligned}
v_1&=(\eta^2-1)\norm{\vw^t-\vw^*-\frac{\gamma}{Mp(i_t)}(\nabla g_{i_t,j}(\vw^t)-\nabla g_{i_t,j}(\vw^*))}_2^2,\\
v_2&=(\eta^2-1)\norm{\frac{\gamma}{Mp(i_t)}\nabla g_{i_t,j}(\vw^*)}_2^2.
\end{aligned}\]
The inequality $(a+b)^2\leq 2a^2+2b^2$ yields that
$v\leq 2v_1+2v_2$.
Similar to \cite[Appendix B]{nguyen2014linear}, we can show that solving $x^2-2ux-v\leq 0$ leads to $x\leq u+\sqrt{u^2+v}$.
Thus the inequality $a+b\leq \sqrt{2a^2+2b^2}$ yields
\[
x^2\leq 2u^2+2(u^2+v)=4u^2+2v\leq 8(u_1^2+u_2^2)+4v_1+4v_2.
\]
By taking the conditional expectation $E_{i_t|I_{t-1}}$ on both sides, we obtain
\[\begin{aligned}
&E_{i_t|I_{t-1}}x^2\leq 8E_{i_t|I_{t-1}}u_1^2+8E_{i_t|I_{t-1}}u_2^2+4E_{i_t|I_{t-1}}(v_1+v_2)\\
&\leq 8(1-(2\gamma-\gamma^2\alpha_{3k,j})\rho_{3k,j}^-)\norm{\vw^t-\vw^*}_2^2
+\frac{8\gamma^2}{\min_{i_t}M^2(p(i_t))^2}E_{i_t|I_{t-1}}\norm{\mathcal{P}_\Omega\nabla g_{i_t,j}(\vw^*)}_2^2\\
&+4(\eta^2-1)(1+\gamma^2\alpha_{3k,j}\bar{\rho}_{3k,j}^+-2\gamma\rho_{3k,j}^-)\norm{\vw^t-\vw^*}_2^2\\
&+\frac{4\gamma^2(\eta^2-1)}{\min_{i_t}M^2(p(i_t))^2}E_{i_t}\norm{\nabla g_{i_t,j}(\vw^*)}_2^2\\
&=\left(8(1-(2\gamma-\gamma^2\alpha_{3k,j})\rho_{3k,j}^-)+4(\eta^2-1)(1+\gamma^2\alpha_{3k,j}\bar{\rho}_{3k,j}^+-2\gamma\rho_{3k,j}^-)\right)\norm{\vw^t-\vw^*}_2^2\\
&+\frac{4\gamma^2}{\min_{i_t}M^2(p(i_t))^2}\left(2E_{i_t}\norm{\mathcal{P}_\Omega\nabla g_{i_t,j}(\vw^*)}_2^2+(\eta^2-1)E_{i_t}\norm{\nabla g_{i_t,j}(\vw^*)}_2^2\right)\\
&:=\kappa_j\norm{\vw^t-\vw^*}_2^2+\sigma_j.
\end{aligned}\]
In the second inequality, the respective upper bounds for $E_{i_t|I_{t-1}}u_1^2$ and $E_{i_t|I_{t-1}}u_2^2$ are obtained by using the Corollary 8 of \cite{nguyen2014linear}. Here $\rho_{k,j}^-$ is defined in \eqref{eqn:rho_Fj}, $\alpha_{k,j}$ and $\bar{\rho}_{k,j}^+$ are defined in \eqref{eqn:alpha_k_j}.
Therefore
\[
E_{I_t}\norm{\vw^{t+1}-\vw^*}_2^2\leq\kappa_jE_{I_{t-1}}\norm{\vw^t-\vw^*}_2^2+\sigma_j.
\]
which implies that
\[
E_{I_t}\norm{X_{\cdot,j}^{t+1}-X_{\cdot,j}^*}_2^2
\leq \kappa_j^{t+1}\norm{X_{\cdot,j}^0-X_{\cdot,j}^*}_2^2+\frac{\sigma_j}{1-\kappa_j},
\]
for $j=1,\ldots,L$.
Here the contraction coefficient is
\begin{equation}\label{eqn:kappaj}
\begin{aligned}
\kappa_j&=8(1-(2\gamma-\gamma^2\alpha_{3k,j})\rho_{3k,j}^-)+4(\eta^2-1)(1+\gamma^2\alpha_{3k,j}\bar{\rho}_{3k,j}^+-2\gamma\rho_{3k,j}^-),
\end{aligned}
\end{equation}
and the tolerance parameter is
\begin{equation}\label{eqn:sigmaj}
\begin{aligned}
\sigma_j&=\frac{4\gamma^2}{\min\limits_{1\leq i\leq M}M^2(p(i_t))^2}\left(2E_{i_t}\norm{\mathcal{P}_\Omega\nabla g_{i_t,j}(\vw^*)}_2^2+(\eta^2-1)E_{i_t}\norm{\nabla g_{i_t,j}(\vw^*)}_2^2\right).
\end{aligned}
\end{equation}
In particular, if $\gamma=\eta=1$ and $p(i)=1/M$, then
\[\begin{aligned}
\kappa_j&=8(1-2\rho_{3k,j}^{-}+\alpha_{3k,j}\rho_{3k,j}^-),\\
\sigma_j&=\frac8M\sum_{i=1}^M\norm{\mathcal{P}_\Omega\nabla g_{i,j}(X_{\cdot,j}^*)}_2^2.
\end{aligned}
\]
\end{proof}

\begin{theorem}[CStoIHT]\label{thm:SMVStoIHT}
Let $X^*$ be a feasible solution of \eqref{eqn:model} and $X^0$ be the initial solution. Under Assumptions, the expectation of the recovery error at the $(t+1)$-th iteration of Algorithm~\ref{alg:CStoIHT} is bounded by
\begin{equation}\label{eqn:SMVStoIHT_E}
E\norm{X^{t+1}-X^*}_F\leq \widehat{\kappa}^{t+1}\norm{X^0-X^*}_F+\widehat{\sigma},
\end{equation}
where $X^{t}$ is the approximation of $X^*$ at the $t$-th iteration of Algorithm~\ref{alg:CStoIHT} with the initial guess $X^0_{\cdot,j}$. Here the contraction coefficient $\widehat{\kappa}$ and the tolerance parameter $\widehat{\sigma}$ are defined as
\[
\widehat{\kappa}=\sqrt{\max_{1\leq j\leq L}\kappa_j},\quad
\widehat{\sigma} = \sqrt{\frac{\sum_{j=1}^L\sigma_j}{1-\max\limits_{1\leq j\leq L}{\kappa_j}}},
\]
where $\kappa_j$ is the contraction coefficient for each StoIHT defined in \eqref{eqn:kappaj}.
\end{theorem}

\begin{proof}
For each $j=1,2,\ldots,L$, StoIHT with the initial guess $X_{\cdot,j}^0$ generates $X^{t+1}_{\cdot,j}$ after $t$ iterations, i.e., the result at the $(t+1)$-th inner iteration and $j$-th outer iteration of Algorithm~\ref{alg:CStoIHT}. Then the expectation of the recovery error squared is bounded by
\[
E\norm{X_{\cdot,j}^{t+1}-X_{\cdot,j}^*}_2^2\leq \kappa_j^{t+1}\norm{X_{\cdot,j}^0-X_{\cdot,j}^*}_2^2+\frac{\sigma_j}{1-\kappa_j},
\]
where $\kappa_j$ and $\sigma_j$ are defined in Lemma~\ref{lem:E2StoIHT}. Note that $\kappa_j$ depends only on the constants in the $\mathcal{D}$-RSC and $\mathcal{D}$-RSS properties of the objective function or its component function $g_{i,j}$ while $\sigma_j$ depends on the feasible solution $X_{\cdot,j}^*$. By combining all $L$ components of $X^{t+1}$, we get
\[
\begin{aligned}
&E\norm{X^{t+1}-X^*}_F\leq\sqrt{E\norm{X^{t+1}-X^*}_F^2}\\
&= \sqrt{\sum_{j=1}^LE\norm{X^{t+1}_{\cdot,j}-X^*_{\cdot,j}}_2^2}\\
&\leq \sqrt{\sum_{j=1}^L\left(\kappa_j^{t+1}\norm{X_{\cdot,j}^0-X_{\cdot,j}^*}_2^2+\frac{\sigma_j}{1-\kappa_j}\right)}\\
&\leq \sqrt{(\max_{1\leq j\leq L}\kappa_j)^{t+1}\norm{X^0-X^*}_F^2+\frac{\sum_{j=1}^L\sigma_j}{1-\max_{1\leq j\leq L}\kappa_j}}\\
&\leq \left(\sqrt{\max\limits_{1\leq j\leq L}\kappa_j}\right)^{t+1}\norm{X^0-X^*}_F+\sqrt{\frac{\sum_{j=1}^L\sigma_j}{1-\max\limits_{1\leq j\leq L}\kappa_j}}\\
&:= \widehat{\kappa}^{t+1}\norm{X^0-X^*}_F+\widehat{\sigma}.
\end{aligned}
\]
In particular, if $\gamma=\eta=1$, then
\[\begin{aligned}
\widehat{\kappa}
&=\sqrt{\max\limits_{1\leq j\leq L}\kappa_j}
=\sqrt{\max\limits_{1\leq j\leq L}8(1-2\rho_{3k,j}^{-}+\alpha_{3k,j}\rho_{3k,j}^-)}\\
&=2\sqrt{2}\sqrt{\max\limits_{1\leq j\leq L}(1-2\rho_{3k,j}^{-}+\alpha_{3k,j}\rho_{3k,j}^-)}.
\end{aligned}
\]
\end{proof}
\noindent\textbf{Remark.} Firstly, MStoIHT can converge without the Assumptions while CStoIHT may not converge to a joint sparse solution without Assumptions, since the joint sparsity is not enforced at each iteration of CStoIHT. Secondly, under Assumptions, if $\rho_{3k,1}^-=\ldots=\rho_{3k,L}^-$ and $\rho_{3k,1}^+=\ldots=\rho_{3k,L}^+$, Lemmas~\ref{lem:rss} and \ref{lem:rsc} yield
$\widehat{\kappa}=\sqrt{2}\kappa$, where $\kappa$ is defined Theorem~\ref{thm1}.

\subsection{MStoGradMP}
By using the same proof techniques as in Theorem~\ref{thm1}, we can get the following convergence result for MStoGradMP.
\begin{theorem}[MStoGradMP]\label{thm2}
Let $X^*$ be a feasible solution of \eqref{eqn:model} and $X^0$ be the initial solution. Assume that $F(X)=\frac1M\sum_{i=1}^Mf_i(X)$ satisfies the $\cD$-RSC with constant $\rho_k^-$ in \eqref{eqn:DRSC} and $f_i$ satisfies $\cD$-RSS with constant $\rho_k^+(i)$ in \eqref{eqn:DRSS}. At the $(t+1)$-th iteration of Algorithm \ref{alg:MStoGradMP}, the expectation of the recovery error is bounded by
\begin{equation}\label{eqn:alg2E}
E\norm{X^{t+1}-X^*}_F\leq \kappa^{t+1}\norm{X^0-X^*}_F+\frac{\sigma_{X^*}}{1-\kappa},
\end{equation}
where
\[\begin{aligned}
\kappa&=(1+\eta_2)\sqrt{\frac{\alpha_{4k}}{\rho_{4k}^-}}
\left(\sqrt{\hat{p}}\sqrt{\frac{\rho_{4k}^+(2\eta_1^2-1)}{\rho_{4k}^-\eta_2^2}-1}+\frac{\sqrt{\eta_1^2-1}}{\eta_1}\right),\\
\sigma_{X^*}&=\frac{(1+\eta_2)}{\rho_{4k}^{-}\tilde{p}}\left(2\hat{p}\sqrt{\frac{\alpha_{4k}}{\rho_{4k}^-}}+3\right)\max_{|\Omega|\leq 4k\atop 1\leq i\leq M}\norm{\mathcal{P}_\Omega\frac{\partial f_i}{\partial X}(X^*)}_F.
\end{aligned}\]
Here $\hat{p}=\max\limits_{1\leq i\leq M}Mp(i)$, $\tilde{p}=\min\limits_{1\leq i\leq M}Mp(i)$ and $\alpha_k,\rho_{k}^+,\bar{\rho}_k^+$ are defined in \eqref{eqn:rho+}. Thus Algorithm \ref{alg:MStoGradMP} converges linearly if $\kappa<1$. In particular, if $\eta_1=\eta_2=1$ and $p(i)=1/M$, then
\begin{equation}\label{eqn:kappa_MStoGradMP}
\kappa=\frac{2\sqrt{\alpha_{4k}(\rho_{4k}^+-\rho_{4k}^-)}}{\rho_{4k}^-}.
\end{equation}
\end{theorem}

\subsection{CStoGradMP}
Similar to CStoIHT, we start the convergence analysis for CStoGradMP by finding the contraction coefficient for the expectation of recovery error squared at each iteration of CStoGradMP.

\begin{lemma}\label{lem:E1StoGradMP}
Let $X^*$ be a feasible solution of \eqref{eqn:bmodel} and $X^0$ be the initial solution. Under Assumptions, the expectation of the recovery error squared at the $t$-th iteration of Algorithm~\ref{alg:CStoGradMP} for estimating the $j$-th column of $X^*$ is bounded by
\begin{equation}\label{eqn:E1StoGradMP}
E_{I_t}\norm{\vb^t-X_{\cdot,j}^*}_2^2\leq \beta_1E_{I_t}\norm{\mathcal{P}_{\hat{\Gamma}}(\vb^t-X_{\cdot,j}^*)}_2^2+\xi_1,
\end{equation}
where
\begin{equation}\label{eqn:beta_delta_1}
\begin{aligned}
\beta_1&=\frac{\alpha_{4k,j}}{2\rho_{4k,j}^{-}-\alpha_{4k,j}},\\
\xi_1&=\frac{2E_{I_t}E_{i}\norm{\mathcal{P}_{\widehat{\Gamma}}\nabla g_{i,j}(X_{\cdot,j})}_2^2}{\alpha_{4k,j}(2\rho_{4k,j}^{-}-\alpha_{4k,j})\min\limits_{1\leq i\leq M} M^2(p(i))^2}.
\end{aligned}\end{equation}
Here $\alpha_{k,j}$ and $\rho_{k,j}^-$ are defined in \eqref{eqn:alpha_k_j}, and $I_t$ is the set of all indices $i_1,\ldots,i_t$ randomly selected at or before the $t$-th step of the algorithm.
\end{lemma}

\begin{proof}
Consider the problem \eqref{eqn:SMVobj}. Following the proof in \cite[Appendix C]{nguyen2014linear}, we can get
\[\begin{aligned}
\norm{\mathcal{P}_{\widehat{\Gamma}}(\vb^t-X_{\cdot,j}^*)}_2^2
&\leq2(1-(2\gamma-\gamma^2\alpha_{4k,j})\rho_{4k,j}^{-})\norm{\vb^t-X_{\cdot,j}^*}_2^2\\
&\phantom{\leq}+\frac{2\gamma^2}{\min\limits_{1\leq i\leq M}M^2(p(i))^2}E_i\norm{\mathcal{P}_{\widehat{\Gamma}}\nabla g_{i,j}(X_{\cdot,j})}_2^2.
\end{aligned}\]
Here we use the inequality $(a+b)^2\leq 2a^2+2b^2$ for $a,b\in\mathbb{R}$ and the expectation inequality $(EX)^2\leq E(X^2)$. Then we have
\[\begin{aligned}
\norm{\vb^t-X_{\cdot,j}^*}_2^2&=\norm{\mathcal{P}_{\widehat{\Gamma}}(\vb^t-X_{\cdot,j}^*)}_2^2
+\norm{\mathcal{P}_{\widehat{\Gamma}^c}(\vb^t-X_{\cdot,j}^*)}_2^2\\
&\leq 2(1-(2\gamma-\gamma^2\alpha_{4k,j})\rho_{4k,j}^{-})\norm{\vb^t-X_{\cdot,j}^*}_2^2\\
&\phantom{\leq}+\frac{2\gamma^2}{\min\limits_{1\leq i\leq M}M^2(p(i))^2}E_i\norm{\mathcal{P}_{\widehat{\Gamma}}\nabla g_{i,j}(X_{\cdot,j}^*)}_2^2+\norm{\mathcal{P}_{\widehat{\Gamma}^c}(\vb^t-X_{\cdot,j}^*)}_2^2.
\end{aligned}\]
Moving the first term on the right hand side to the left hand side leads to
\[
\begin{aligned}
\norm{\vb^t-X_{\cdot,j}^*}_2^2&
\leq \frac{2\gamma^2}{\phi\min\limits_{1\leq i\leq M}M^2(p(i))^2}E_i\norm{\mathcal{P}_{\widehat{\Gamma}}\nabla g_{i,j}(X_{\cdot,j})}_2^2+\frac1\phi\norm{\mathcal{P}_{\widehat{\Gamma}^c}(\vb^t-X_{\cdot,j}^*)}_2^2,
\end{aligned}
\]
where $\phi=2\rho_{4k,j}^{-}(2\gamma-\gamma^2\alpha_{4k,j})-1$. Maximizing $\phi$ with respect to $\gamma$ yields $\gamma=1/\alpha_{4k,j}$ and $\phi_{\max}=(2\rho_{4k,j}^{-}-\alpha_{4k,j})/\alpha_{4k,j}$. By choosing the optimal value of $\gamma$ and taking the expectation with respect to $I_t$ on the both sides of the above inequality, we get \eqref{eqn:E1StoGradMP}.
\end{proof}

Similarly, using the inequality $EX\leq \sqrt{E(X^2)}$ and the fact that $a\leq b+c$ yields $a^2\leq 2b^2+2c^2$, we are able to get the following result, which is different from Lemma 3 in \cite{nguyen2014linear} in that we consider the expectation for the $\ell_2$-norm squared here rather than that for the $\ell_2$-norm.
\begin{lemma}\label{lem:E2StoGradMP}
Let $X^*$ be a feasible solution of \eqref{eqn:bmodel} and $X^0$ be the initial solution. Under Assumptions, the expectation of the recovery error squared at the $t$-th iteration of Algorithm~\ref{alg:CStoGradMP} for estimating the $j$-th column of $X^*$ is bounded by
\begin{equation}\label{eqn:E2StoGradMP}
E_{i_t}\norm{\mathcal{P}_{\widehat{\Gamma}^c}(\vb^t-X_{\cdot,j}^*)}_2^2
\leq\beta_2\norm{X_{\cdot,j}^t-X_{\cdot,j}^*}_2^2+\xi_2,
\end{equation}
where $i_t$ is the index randomly selected at the $t$-th iteration of the CStoGradMP and
\begin{equation}\label{eqn:beta_delta_2}
\begin{aligned}
\beta_2&=4\max\limits_{1\leq i\leq M}Mp(i)\frac{(2\eta_1^2-1)\rho_{4k,j}^{+}-\eta_1^2\rho_{4k,j}^{-}}{\eta_1^2\rho_{4k,j}^{-}}+\frac{2(\eta_1^2-1)}{\eta_1^2},\\
\xi_2&=8\left(\frac{\max\limits_{1\leq i\leq M}p(i)}{\rho_{4k,j}^{-}\min\limits_{1\leq i\leq M}p(i)}\right)^2\max\limits_{|\Omega|\leq 4k\atop 1\leq i\leq M}\norm{\mathcal{P}_{\Omega}\nabla g_{i,j}(X^*_{\cdot,j})}_2^2.
\end{aligned}
\end{equation}
\end{lemma}

\begin{theorem}[CStoGradMP]\label{thm:SMVStoGradMP}
Let $X^*$ be a feasible solution of \eqref{eqn:bmodel} and $X^0$ be the initial solution. Under Assumptions, at the $(t+1)$-th iteration of Algorithm~\ref{alg:CStoGradMP}, there exist $\widetilde{\kappa},\widetilde{\sigma}>0$ such that the expectation of the recovery error is bounded by
\begin{equation}\label{eqn:SMVStoGradMP_E}
E\norm{X^{t+1}-X^*}_F\leq \widetilde{\kappa}^{t+1}\norm{X^0-X^*}_F+\widetilde{\sigma},
\end{equation}
where $X^{t}_{\cdot,j}$ is the approximation of $X_{\cdot,j}^*$ at the $t$-th iteration of CStoGradMP with the initial guess $X^0_{\cdot,j}$.
\end{theorem}

\begin{proof}
At the $t$-th iteration of Algorithm~\ref{alg:CStoGradMP}, we have
\[
\norm{X_{\cdot,j}^{t+1}-\vb^t}_2^2\leq\eta_2^2\norm{\vb_{(k)}^t-\vb^t}_2^2\leq\eta_2^2\norm{X_{\cdot,j}^*-\vb^t}_2^2,
\]
where $\vb^t_{(k)}$ is the best $k$-sparse approximation of $\vb^t$ with respect to the atom set $\mathcal{D}$. Therefore, we get
\[\begin{aligned}
\norm{X_{\cdot,j}^{t+1}-X_{\cdot,j}^*}_2^2
&\leq \norm{X_{\cdot,j}^{t+1}-\vb^t+\vb^t-X_{\cdot,j}^*}_2^2\\
&\leq 2\norm{X_{\cdot,j}^{t+1}-\vb^t}_2^2+2\norm{\vb^t-X_{\cdot,j}^*}_2^2\\
&\leq (2+2\eta_2^2)\norm{\vb^t-X_{\cdot,j}^*}_2^2.
\end{aligned}\]
Next we establish the relationships among various expectations
\[
\begin{aligned}
E_{I_t}\norm{X_{\cdot,j}^{t+1}-X_{\cdot,j}}_2^2
&\leq (2+2\eta_2^2)E_{I_t}\norm{\vb^t-X_{\cdot,j}^*}_2^2\\
&\leq (2+2\eta_2^2)\left(\beta_1E_{I_t}\norm{\mathcal{P}_{\hat{\Gamma}}(\vb^t-X_{\cdot,j}^*)}_2^2+\xi_1\right)\\
&\leq (2+2\eta_2^2)\beta_1\left(\beta_2E_{I_t}\norm{X_{\cdot,j}^t-X_{\cdot,j}^*}_2^2+\xi_2\right)+(2+2\eta_2^2)\xi_1\\
&:=\kappa_j\norm{X_{\cdot,j}^t-X_{\cdot,j}^*}_2^2+\sigma_j,
\end{aligned}
\]
where the first inequality is guaranteed by Lemma~\ref{lem:E1StoGradMP} and the second inequality is guaranteed by Lemma~\ref{lem:E2StoGradMP}.
Here the contraction coefficient $\kappa_j$ and the tolerance parameter $\sigma_j$ are defined by
\[
\kappa_j=(2+2\eta_2^2)\beta_1\beta_2,\quad
\sigma_j=(2+2\eta_2^2)\beta_1\xi_2+(2+2\eta_2^2)\xi_1,
\]
where $\beta_1,\xi_1$ are defined in \eqref{eqn:beta_delta_1} and $\beta_2,\xi_2$ are defined in \eqref{eqn:beta_delta_2}.
Then similar to the proof of Theorem~\ref{thm:SMVStoIHT}, we can derive that
\[
E\norm{X^{t+1}-X^*}_F\leq \widetilde{\kappa}^{t+1}\norm{X^0-X^*}_F+\widetilde{\sigma}
\]
where
\[
\widetilde{\kappa}=\sqrt{\max\limits_{1\leq j\leq L}\kappa_j},\quad \mbox{and}\quad
\widetilde{\sigma} = \sqrt{\frac{\sum_{j=1}^L\sigma_j}{1-\max\limits_{1\leq j\leq L}{\kappa_j}}}.
\]
In particular, if $\eta_1=\eta_2=1$, $p(i)=1/M$, $\alpha_{4k,j}=\alpha_{4k}$, $\rho_{4k,j}^+=\rho_{4k}^+$ and $\rho_{4k,j}^-=\rho_{4k}^-$ for $j=1,\ldots,L$, then we have
\[
\widetilde{\kappa}=4\sqrt{\frac{\alpha_{4k}(\rho_{4k}^{+}-\rho_{4k}^{-})}{\rho_{4k}^{-}(2\rho_{4k}^{-}-\alpha_{4k})}}.
\]
If, in addition, $\alpha_{4k}=\rho_{4k}^-$, then we have $\widetilde{\kappa}=2\kappa$,
where the contraction coefficient $\kappa$ for MStoGradMP is given in \eqref{eqn:kappa_MStoGradMP}, which implies that MStoGradMP converges faster than CStoGradMP in this case due to the smaller contraction coefficient. Compared with MStoIHT, MStoGradMP has even larger convergence improvement in terms of recovery accuracy and running time.
\end{proof}

\noindent\textbf{Remark.}
Convergence analysis in this section provides useful theoretical tools to compare the proposed MMV stochastic greedy algorithms and their SMV counterparts. First, we provide general guarantees via contraction coefficients for all proposed algorithms in Theorems \ref{thm1}, \ref{thm:SMVStoIHT}, \ref{thm2} and \ref{thm:SMVStoGradMP}. Since the contraction coefficient controls the convergence speed of each algorithm, it can provide theoretical guidance on parameter selection in such type of stochastic greedy algorithms. Second, the complicated form of $\kappa$ and $\hat{\kappa}$ can be simplified in some special cases, e.g., the case when $\eta=1$. Lastly, the relationship $\kappa=\hat{\kappa}/\sqrt{2}$ indicates why MStoIHT/MStoGradMP performs better than CStoIHT/CStoGradMP in expectation.

\section{Distributed Compressive Sensing Application}\label{sec:DCS}
In this section, we show that the objective function commonly used in the distributed compressive sensing problem satisfies the $\mathcal{D}$-RSC and $\mathcal{D}$-RSS properties, which paves the theoretical foundation for using the proposed algorithms in this application.
Suppose that there are $L$ underlying signals $\vx_j\in\mathbb{R}^n$ for $j=1,2,\ldots,L$, and their measurements are generated by
\[
\vy_j=A^{(j)} \vx_j,\quad j=1,2,\ldots,L
\]
where $A^{(j)}\in\mathbb{R}^{m\times n}$ ($m\ll n$) is the measurement matrix (a.k.a. the sensing matrix). For discussion simplicity, we assume all measurement matrices are the same, i.e., $A^{(j)}=A=[A_{\cdot,1},\ldots,A_{\cdot,n}]$. By concatenating all vectors as a matrix, we rewrite the above equation as $Y=AX$ where
$Y=[\vy_1,\ldots,\vy_L]\in\mathbb{R}^{m\times L}$, and $X=[\vx_1,\ldots,\vx_L]\in\mathbb{R}^{n\times L}$.

Now assume that the atom set is finite and denote $\mathcal{D}=\{\vd_1,\ldots,\vd_{N}\}$ with the corresponding dictionary $D=[\vd_1,\ldots,\vd_{N}]$. Consider the following distributed compressive sensing model with common sparse supports \cite{baron2006distributed}
\begin{equation}\label{eqn:CSmodel}
\begin{aligned}
&\min_{X}\frac1{2m}\sum_{j=1}^L\norm{\vy_j-A \vx_j}_2^2\\
&\qquad\st\quad \vx_j=D\theta_j\quad\supp(\theta_j)=\Omega\subseteq\{1,2,\ldots,N\}.
\end{aligned}\end{equation}
Here the objective function has the form
\begin{equation}\label{eqn:CSF}
F(X)=\frac1{2m}\norm{Y-A X}_F^2.
\end{equation}
Then $F(X)$ can be written as
$F(X)=\frac1M\sum_{i=1}^Mf_i(X)$,
where $M=m/b$ and
\begin{equation}\label{eqn:CSfi}
f_i(X)=\frac{1}{2b}\sum_{j=1}^L(Y_{i,j}-\sum_{k=1}^nA_{i,k}X_{k,j})^2=\frac{1}{2b}\norm{Y_{i,\cdot}-A_{i,\cdot}X}_2^2.
\end{equation}
The above expression shows that $f_i$'s satisfy the Assumption (a) and thereby the concatenated algorithms in Section~\ref{sec:algo} can be applied.
We first compute the partial derivative. For $s=1,2,\ldots,n$ and $t=1,2,\ldots,L$, we have
\[
\begin{aligned}
\frac{\partial f_i(X)}{\partial X_{s,t}}&
=\frac{1}{2b}\sum_{j=1}^L2\left(\sum_{k=1}^nA_{i,k}X_{k,j}-Y_{i,j}\right)\sum_{k=1}^nA_{i,k}\frac{\partial X_{k,j}}{\partial X_{s,t}}\\
&=\frac1b\sum_{j=1}^L\left(\sum_{k=1}^nA_{i,k}X_{k,j}-Y_{i,j}\right)\sum_{k=1}^nA_{i,k}\delta_{k,s}\delta_{j,t}\\
&=\frac1b\left(\sum_{k=1}^nA_{i,k}X_{k,t}-Y_{i,t}\right)A_{i,s}.
\end{aligned}
\]
Here $\delta_{i,j}=1$ if $i=j$ and zero otherwise. Thus the generalized gradient of $f_i(X)$ with respect to $X$ has the form
\[
\frac{\partial f_i(X)}{\partial X}=\frac1b\begin{bmatrix}\dfrac{\partial f_i(X)}{\partial X_{s,t}}\end{bmatrix}_{st}
=\frac1bA_{i,\cdot}^T(A_{i,\cdot}X-Y_{i,\cdot}).
\]

\begin{lemma}\label{lem1}
If the sensing matrix $A\in\mathbb{R}^{m\times n}$ satisfies the Restricted Isometry Property (RIP), i.e., there exists $\delta_k>0$ such that
\[
(1-\delta_k)\norm{\vx}_2^2\leq \norm{A\vx}_2^2\leq(1+\delta_k)\norm{\vx}_2^2
\]
for any $k$-sparse vector $\vx\in\mathbb{R}^n$, then the function $F(X)$ defined in \eqref{eqn:CSF} satisfies the $\mathcal{D}$-restricted strong convexity property.
\end{lemma}

\begin{proof}
Let $X\in\mathbb{R}^{n\times L}$ with $k$ nonzero rows, which implies that each column of $X$ has at most $k$ nonzero components. By the RIP of $A$, we have
\[
(1-\delta_k)\norm{X_{\cdot,j}}_2^2\leq \norm{A X_{\cdot,j}}_2^2\leq (1+\delta_k)\norm{X_{\cdot,j}}_2^2,
\]
for $j=1,\ldots,L$. Note that $\norm{X}_F^2=\sum_{j=1}^L\norm{X_{\cdot,j}}_2^2$. Thus we get
\[
(1-\delta_k)\norm{X}_F^2\leq \norm{A X}_F^2\leq (1+\delta_k)\norm{X}_F^2.
\]
For any two $X,X'\in\mathbb{R}^{n\times L}$ with $|\supp_{\mathcal{D}}^r(X)\cup\supp_{\mathcal{D}}^r(X')|\leq k$, we have
\[
\begin{aligned}
&\phantom{=}F(X')-F(X)-\Big\langle \frac{\partial F(X)}{\partial X},X'-X\Big\rangle \\
&=\frac1{2m}\left(\norm{Y-A X'}_F^2-\norm{Y-A X}_F^2\right)-\Big\langle \frac1mA^T(A X-Y),X'-X\Big\rangle\\
&=\frac1{2m}\norm{A(X'-X)}_F^2\geq\frac{1-\delta_k}{2m}\norm{X'-X}_F^2.
\end{aligned}
\]
Thus $F(X)$ satisfies the $\mathcal{D}$-restricted strong convexity property with $\rho^{-}_k=\frac{1-\delta_k}{2m}$.
\end{proof}

\begin{lemma}\label{lem2}
If the sensing matrix $A\in\mathbb{R}^{m\times n}$ satisfies the following property: for any $k$-sparse vector $\vx\in\mathbb{R}^n$, there exists $\delta_k>0$ such that
\[
\frac1b\norm{A_{\tau_i,\cdot}^TA_{\tau_i,\cdot}\vx}_2\leq(1+\delta_k)\norm{\vx}_2
\]
where $A_{\tau_i,\cdot}$ is formed by extracting rows of $A$ with row indices in the $i$-th batch index set $\tau_i$.
Then the function $f_i(X)$ defined in \eqref{eqn:CSfi} satisfies the $\mathcal{D}$-restricted strong smoothness property.
\end{lemma}

\begin{proof}
Let $X\in\mathbb{R}^{n\times L}$ have $k$ nonzero rows. Then for $j=1,\ldots,L$, we have
\[
\frac1b\norm{A_{\tau_i,\cdot}^TA_{\tau_i,\cdot}X_{\cdot,j}}_2\leq(1+\delta_k)\norm{X_{\cdot,j}}_2,
\]
which implies that
\[
\frac1b\norm{A_{\tau_i,\cdot}^TA_{\tau_i,\cdot}X}_F\leq(1+\delta_k)\norm{X}_F.
\]
For any two $X,X'\in\mathbb{R}^{n\times L}$ with $|\supp_{\mathcal{D}}^r(X)\cup\supp_{\mathcal{D}}^r(X')|\leq k$, we have
\[
\begin{aligned}
&\phantom{=}\norm{\frac{\partial f_i(X)}{\partial X}-\frac{\partial f_i(X')}{\partial X}}_F\\
&=\frac1b\norm{A_{i,\cdot}^T(A_{i,\cdot}X-Y_{i,\cdot})-A_{i,\cdot}^T(A_{i,\cdot}X'-Y_{i,\cdot})}_F\\
&=\frac1b\norm{A_{i,\cdot}^TA_{i,\cdot}(X-X')}_F\leq (1+\delta_k)\norm{X-X'}_F.
\end{aligned}
\]
Therefore $f_i(X)$ satisfies the $\mathcal{D}$-restricted strong convexity with $\rho^+_k(i)={1+\delta_k}$.
\end{proof}

By Lemma~\ref{lem1}, Lemma~\ref{lem2} and the convergence analysis in Section~\ref{sec:convergence}, the contraction coefficient in the proposed algorithms depends on the coefficient in the RIP condition, whose infimum for some special type of matrices are available \cite{foucart2013mathematical}. In addition, the convergence guarantees of all proposed algorithms in terms of the RIP constant for the distributed compressive sensing can be obtained by plugging the expression of $\rho^+_k$ and $\rho^-_k$ in the proofs of Lemmas \ref{lem1} and \ref{lem2} into $\kappa$, $\widehat{\kappa}$ and $\widetilde{\kappa}$ in Theorems \ref{thm1}, \ref{thm:SMVStoIHT}, \ref{thm2}, and \ref{thm:SMVStoGradMP}.

\section{Numerical Experiments}\label{sec:exp}
In this section, we conduct a variety of numerical experiments to validate the effectiveness of the proposed algorithms. More specifically, our tests include reconstruction of row sparse signals from a linear system and joint sparse video sequence recovery. To compare different results quantitatively, we use the relative error defined as
$\mbox{ReErr}={\norm{X^t-X^*}_F}/{\norm{X^*}_F}$,
where $X^*$ is the ground truth and $X^t$ is the estimation of $X^*$ at the $t$-th iteration. Regarding the computational efficiency, we also record the running time which counts all the computation time over a specified number of iterations excluding data loading or generation. Here we record the running time by using the commands \verb"tic" and \verb"toc" in Matlab. To assess the concatenated SMV algorithms, i.e., CStoIHT and CStoGradMP, we apply the SMV algorithm sequentially to the same sensing matrix and all columns of the measurement matrix $Y$, and save all intermediate approximations of each column of $X$ for further computation of the relative error. In all tests, we use the discrete uniform distribution, i.e., $p(i)=1/M$ for $i=1,2,\ldots,M$ in the non-batched versions and $p(i)=1/d$ for $i=1,2,\ldots,d$ in the batched versions. The parameter $\eta$ is fixed as 1. By default, each algorithm is stopped when either the relative error between two subsequent approximations of $X^*$ reaches the tolerance or the maximum number of iterations is achieved. All our experiments are performed in a desktop with an Intel\textsuperscript{\textregistered} Xeon\textsuperscript{\textregistered} CPU E5-2650 v4 @ 2.2GHz and 64GB RAM in double precision. The algorithms are implemented in Matlab 2016a running on Windows 10.

\subsection{Joint Sparse Matrix Recovery}
In the first set of experiments, we compare the proposed algorithms and their concatenated SMV counterparts in terms of reconstruction error and running time. In particular, we investigate the impact of the sparsity level $k$, and the number of underlying signals $L$ to be reconstructed on the performance of BCStoIHT, BMStoIHT, BCStoGradMP and BMStoGradMP, in terms of relative error and the running time. To reduce randomness in the results, we run 50 trials for each test with fixed parameters and then take the average over the number of trials. Note that since different trials may take different number of iterations to reach the desired accuracy, we only count the common iterations when comparing the reconstruction error and the runtime.

First, we compare CStoIHT and MStoIHT in both non-batched and batched versions, and fix the maximum number of iterations as 1000, the stopping criteria tolerance $\varepsilon=10^{-6}$ and $\gamma=1$ in both algorithms. To start with, we create a sensing matrix $A\in\mathbb{R}^{100\times 200}$ where each entry follows the normal distribution with zero mean and variance of $1/100$ and each column of $A$ is normalized by dividing its $\ell_2$-norm. In this way, it can be shown that the spark of $A$, i.e., the smallest number of linearly dependent columns of $A$, is 100 with probability one \cite{Elad2010}. To create a signal matrix $X^*\in\mathbb{R}^{200\times 40}$, we first generate a Gaussian distributed random matrix of size $200\times 40$, and then randomly zero out $(200-k)$ rows where $k$ is the row sparsity of $X^*$. The measurement matrix $Y$ is created by $AX$ for the noise-free cases. By choosing the sparsity level $k\in\{10,20\}$ and the batch size $b\in\{1,10\}$, we obtain the results shown in Figure~\ref{fig:StoIHT:sCmp}. Since the initial guess for the signal matrix is set a zero matrix, all the error curves start with the point $(0,1)$. Notice that to show the computational efficiency, we use the running time in seconds as the horizonal axis rather than the number of iterations. It can be seen that as the sparsity level grows, i.e., the signal matrix is less joint sparse, more running time (or iterations) is required to achieve the provided tolerance in terms of the relative reconstruction error. Meanwhile, as the batch size increases, BMStoIHT performs better than the sequential BCStoIHT. With large sparsity levels, the inaccurate joint support obtained in the concatenated SMV algorithms cause large relative errors in the first few iterations (see Figure~\ref{fig:StoIHT:sCmp}).

Next, we fix the sparsity level $k$ as 5 and choose the number of signals as $L\in\{40, 80\}$. Figure~\ref{fig:StoIHT:LCmp} compares the results obtained by BCStoIHT and BMStoIHT when the batch size is 1 and 10. In general, BMStoIHT takes less running time than its sequential concatenated SMV counterpart. We can see that mini-batching significantly improves the reconstruction accuracy and reduces the running time of BMStoIHT. After a large number of tests, we also find that the computational time of BMStoIHT is almost linear with respect to the number of signals to be reconstructed. More numerical experiments can be found in our more recent work on the hyperspectral diffuse optical imaging \cite{durgin2019fast}. Lastly, to test the robustness to noise, we add the Gaussian noise with zero mean and standard deviation (a.k.a. noise level) $\sigma\in\{0.04,0.08\}$ to the measurement matrix $Y$. The relative errors for all BCStoIHT and BMStoIHT results versus running time are shown in Figure~\ref{fig:StoIHT:noise}. It is worth noting that the change of sparsity and noise levels have insignificant impact on the running time, which explains that the curve corresponding to the same algorithm stops almost at the same horizontal coordinate in Figure~\ref{fig:StoIHT:sCmp} and Figure~\ref{fig:StoIHT:noise}. By contrast, the running time grows as the number of signals to be recovered increases which suggests that the endpoint of each curve has different horizontal coordinates in Figure~\ref{fig:StoIHT:LCmp}.

\setlength{\tabcolsep}{0pt}
\begin{figure}
\centering
\begin{tabular}{cc}
\includegraphics[width=.46\textwidth]{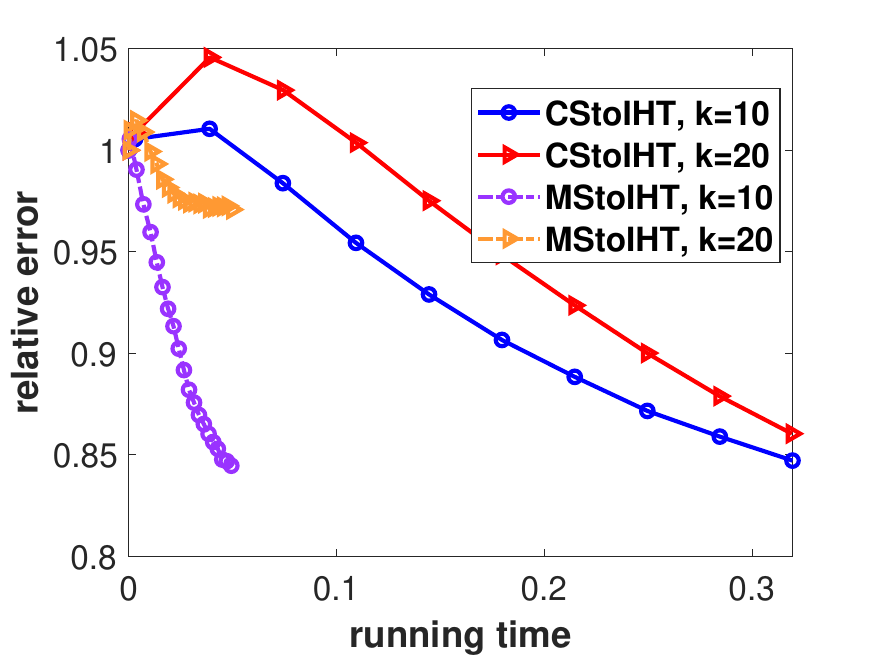}&
\includegraphics[width=.46\textwidth]{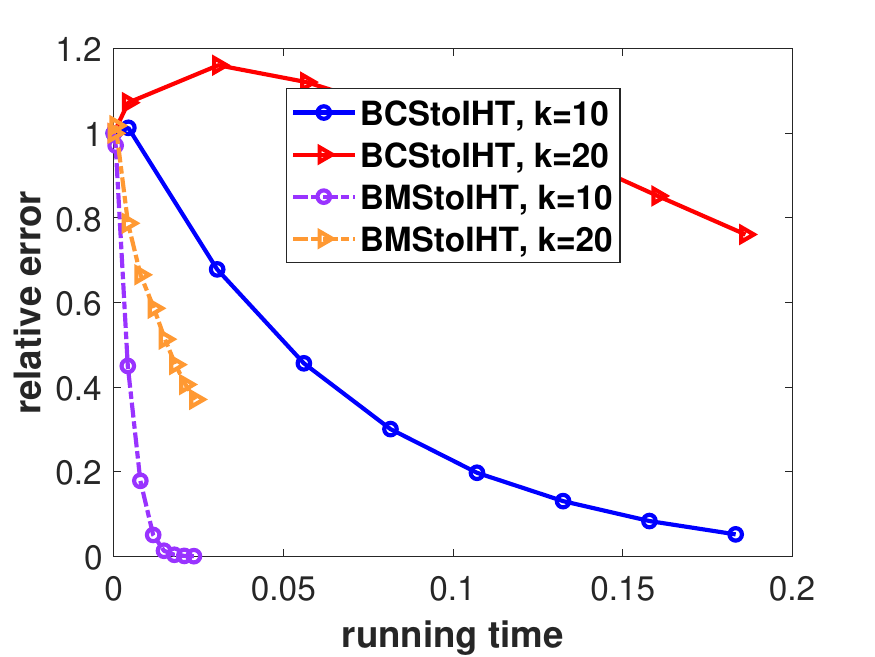}
\end{tabular}
\caption{Comparison of BCStoIHT and BMStoIHT for various sparsity levels of the signal matrix. From left to right: batch sizes are 1 and 10.}\label{fig:StoIHT:sCmp}
\end{figure}

\begin{figure}
\centering
\begin{tabular}{cc}
\includegraphics[width=.46\textwidth]{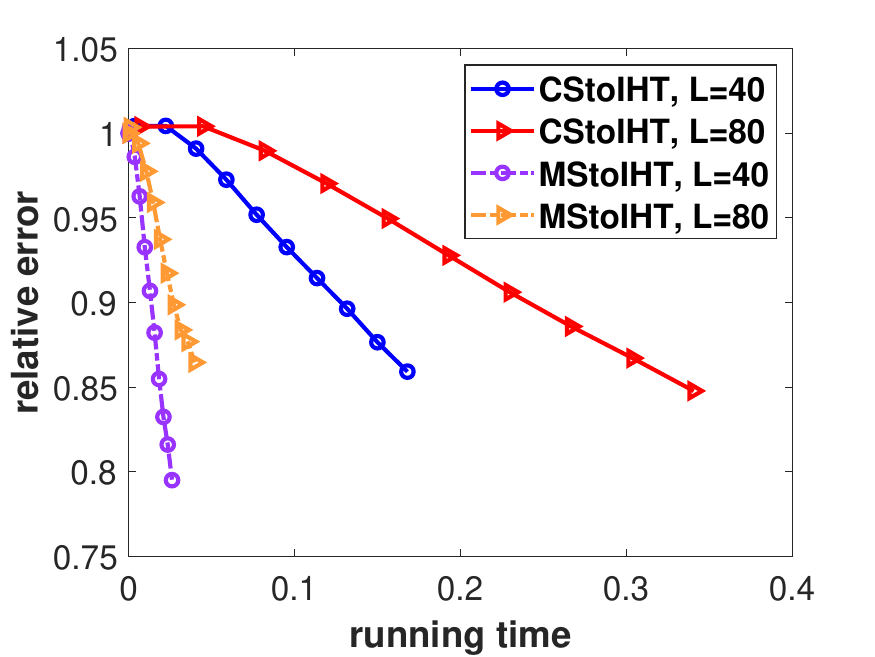}&
\includegraphics[width=.46\textwidth]{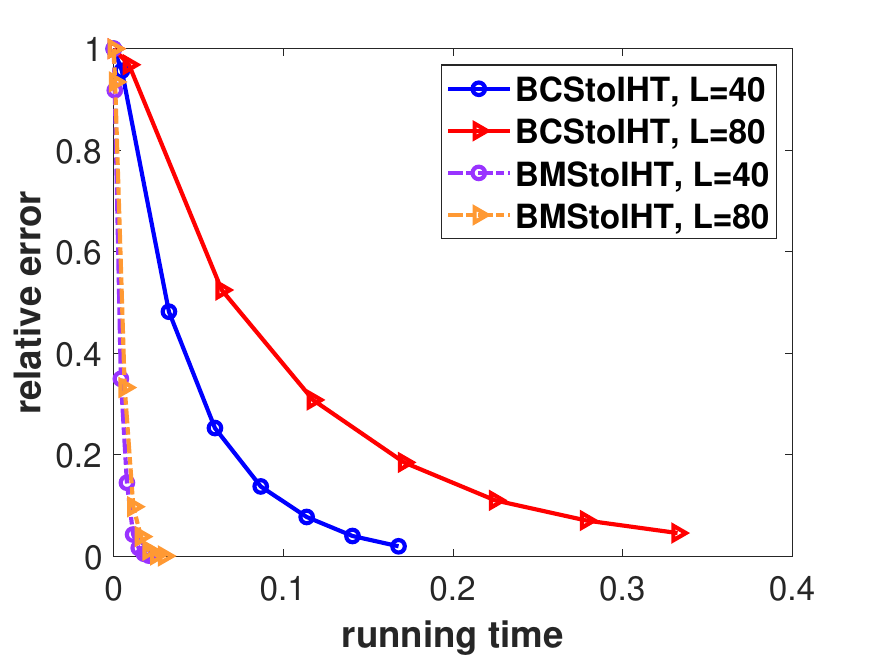}
\end{tabular}
\caption{Comparison of BCStoIHT and BMStoIHT for various numbers of signals to be reconstructed. From left to right: batch sizes are 1 and 10.}\label{fig:StoIHT:LCmp}
\end{figure}

\begin{figure}
\centering
\begin{tabular}{cc}
\includegraphics[width=.46\textwidth]{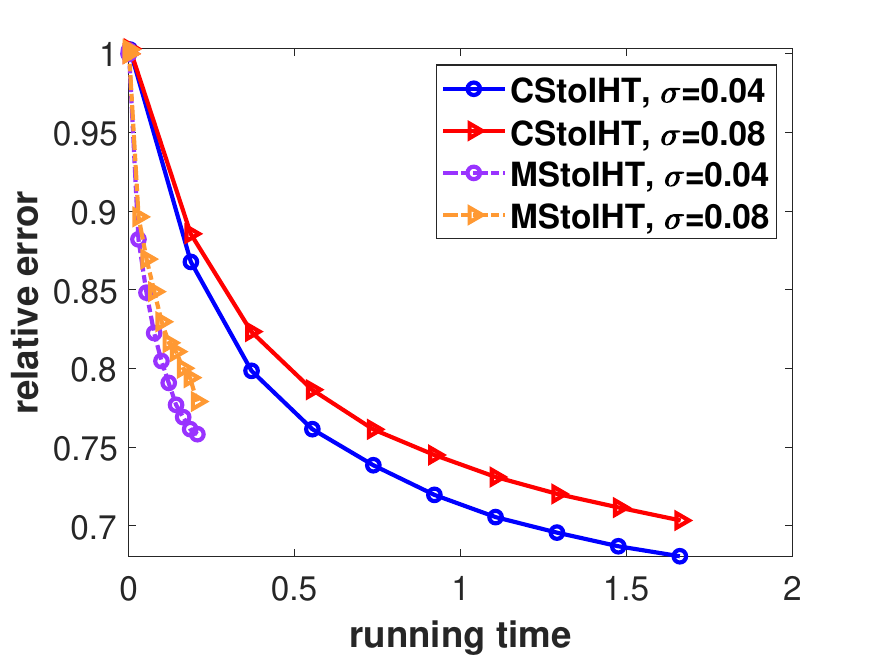}&
\includegraphics[width=.46\textwidth]{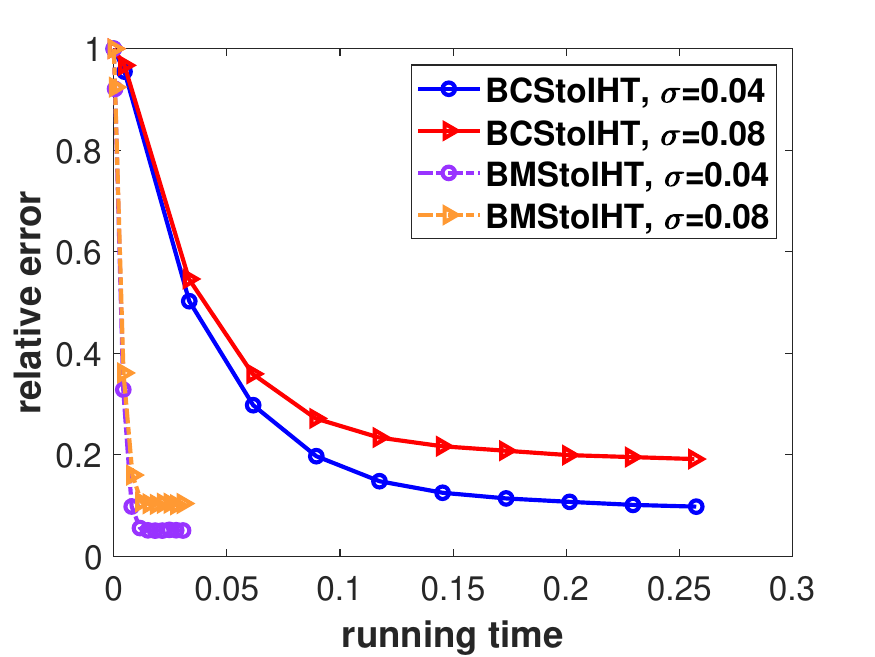}
\end{tabular}
\caption{Comparison of BCStoIHT and BMStoIHT for various noise levels to the measurement matrix. From left to right: batch sizes are 1 and 10.}\label{fig:StoIHT:noise}
\end{figure}

In the second set of tests, we compare CStoGradMP and MStoGradMP in non-batched and batched versions. It is known that StoGradMP usually converges much faster than StoIHT, which is also true for the developed MMV versions. We fix the maximum number of iterations as 30, the stopping criteria tolerance $\varepsilon=10^{-5}$, $\gamma=1$, and the batch size as 1 (non-batched version). Similar to the previous tests, we create a $100\times 200$ random matrix whose columns are normalized, and fix the number of signals $L=40$.  Figure~\ref{fig:StoGradMP:sCmp} shows the results obtained by CStoGradMP and MStoGradMP with sparsity level $k\in\{70,90\}$. Note that for better visualization, we skip the starting point (0,1) for all relative error plots, and use the base 10 logarithmic scale for the horizontal axis of running time since the MStoGradMP takes much less running time than CStoGradMP after the same number of iterations. Unlike StoIHT and MStoIHT, both CStoGradMP and MStoGradMP require that the sparsity level is no more than $n/2$, i.e., 100 in our case. As the sparsity $k$ increases, the operator \verb"pinv" for computing the pseudo-inverse matrix becomes more computationally expensive for matrices with more columns than their rank, which results in the significant growth of running time. For sparse signal matrices, StoGradMP performs better than StoIHT in terms of convergence. Next, we set the number of signals as $20,80$, and get the results shown in Figure~\ref{fig:StoGradMP:LCmp}. It can be seen that MStoGradMP always takes less running time with even higher accuracy than the sequential CStoGradMP. We also discovered that the computation speedup of MStoIHT is almost constant with respect to the number of signals to be reconstructed. The robustness comparison is shown in Figure~\ref{fig:StoGradMP:noise}, where the noise level ranges in $\{0.04,0.08\}$. Furthermore, it is empirically shown that the BMStoGradMP performs much better than BMStoIHT considering their respective convergence behavior and robustness.

\begin{figure}
\centering
\begin{tabular}{cc}
\includegraphics[width=.46\textwidth]{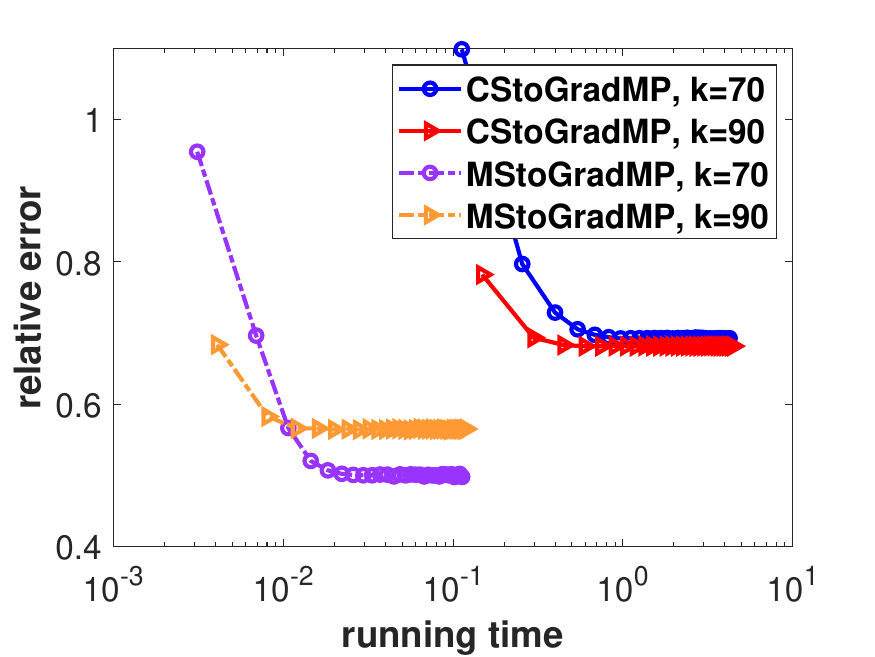}&
\includegraphics[width=.46\textwidth]{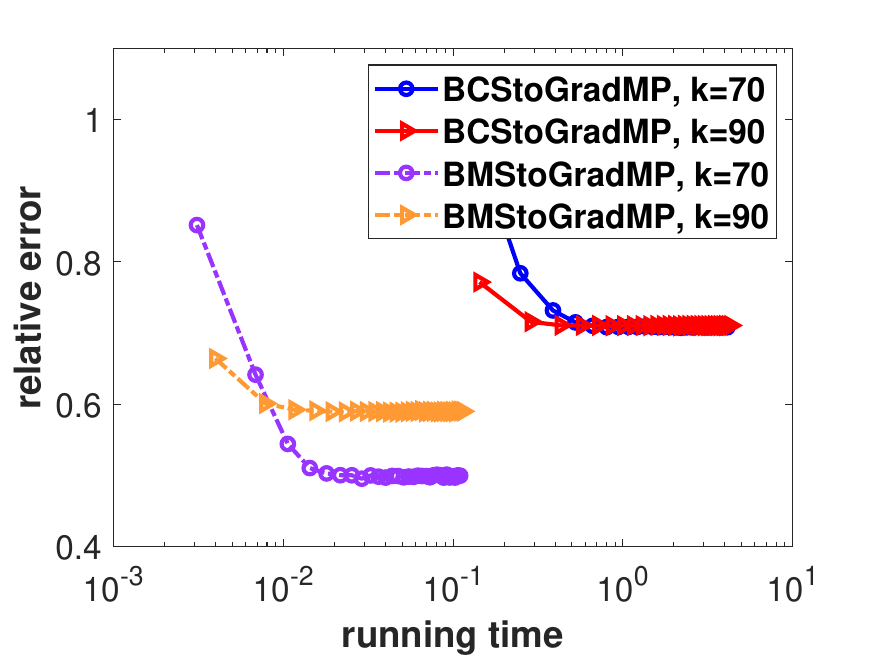}
\end{tabular}
\caption{Comparison of BCStoGradMP and BMStoGradMP with various sparsity levels. From left to right: batch sizes are 1 and 10.}\label{fig:StoGradMP:sCmp}

\end{figure}

\begin{figure}
\centering
\begin{tabular}{cc}
\includegraphics[width=.46\textwidth]{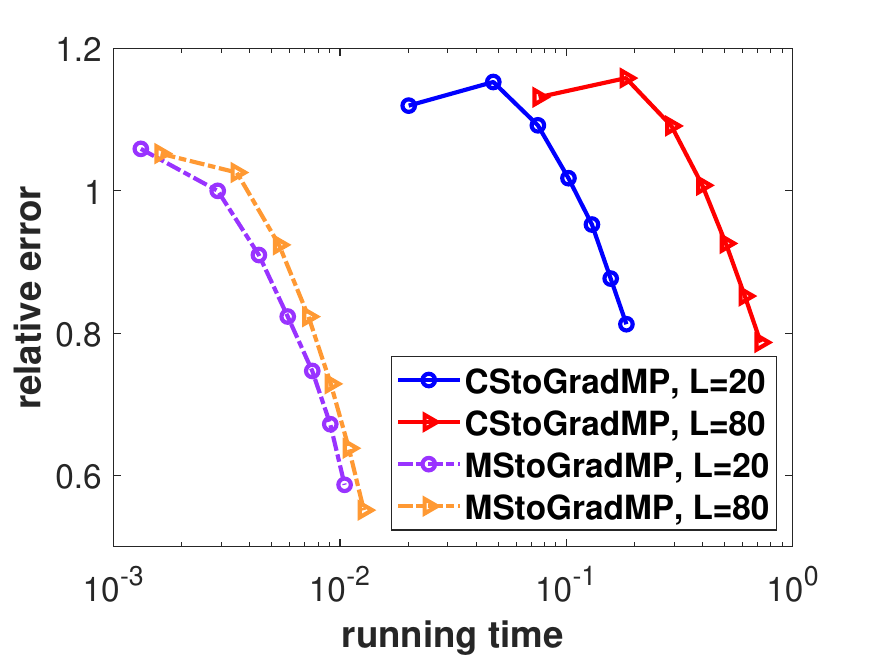}&
\includegraphics[width=.46\textwidth]{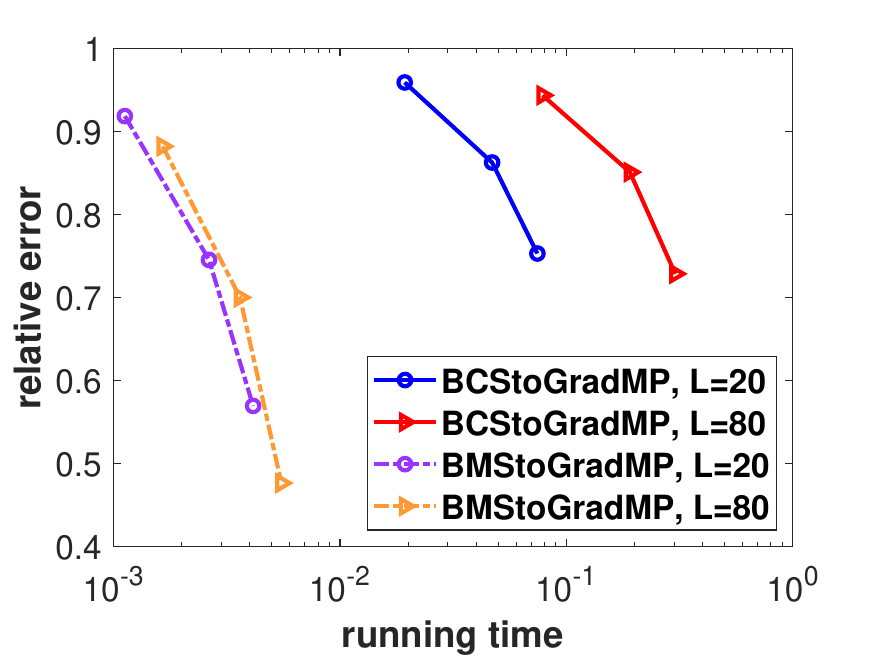}
\end{tabular}
\caption{Comparison of BCStoGradMP and BMStoGradMP with various numbers of signals. From left to right: batch sizes are 1 and 10.}\label{fig:StoGradMP:LCmp}

\end{figure}

\begin{figure}
\centering
\begin{tabular}{cc}
\includegraphics[width=.46\textwidth]{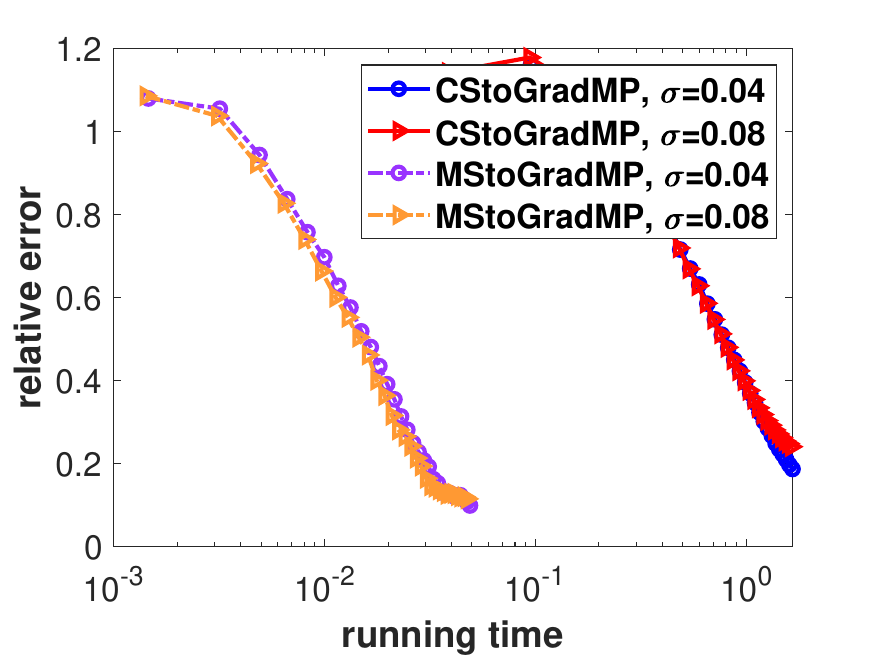}&
\includegraphics[width=.46\textwidth]{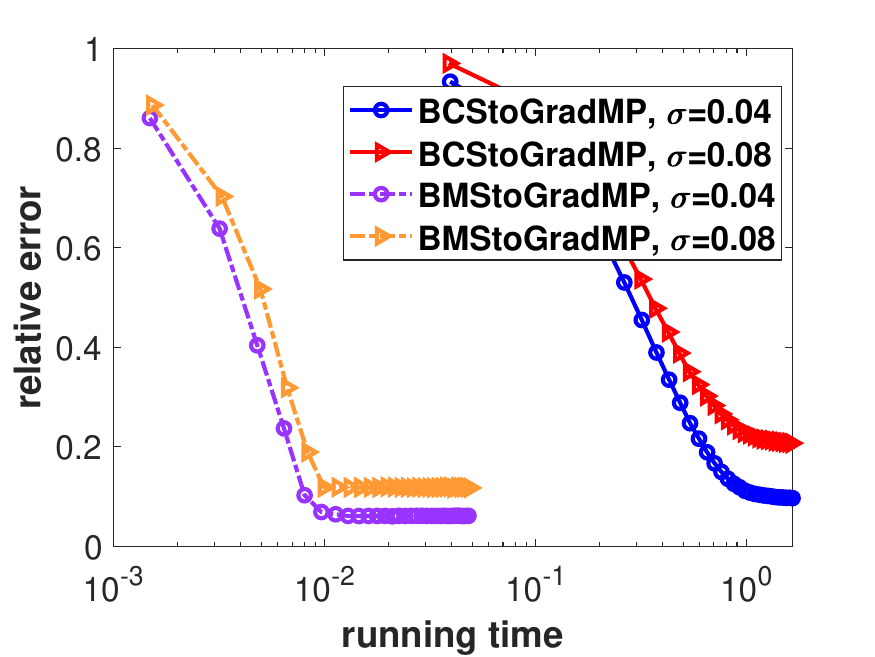}
\end{tabular}
\caption{Comparison of BCStoGradMP and BMStoGradMP with various noise levels. From left to right: batch sizes are 1 and 10.}\label{fig:StoGradMP:noise}

\end{figure}

Lastly, we compare our proposed BMStoIHT and BMStoGradMP with the M-FOCUSS \cite{cotter2005sparse}, and state-of-the-art MMV greedy algorithms \cite{blanchard2014greedy} including M-OMP, M-SP, M-CoSaMP and M-HTP, using the previous noise-free data set with $k=10$, $L=40$, the maximum number of iterations as 100 and the stopping criteria tolerance $\varepsilon=10^{-5}$. The batch size for both BMStoIHT and BMStoGradMP is fixed as 20. The reconstruction errors and running times for all methods are listed in Table~\ref{tab:exp1}, which shows that BMStoGradMP can achieve comparable reconstruction accuracy as M-OMP, M-SP, M-CoSaMP and M-HTP but with the least running time.

\begin{table}
\centering
\begin{tabular}{p{1in}p{1in}c}
\hline\hline
Method & Error & Runtime (sec.)\\ \hline
BMStoIHT & $5.64\times10^{-7}$ & 0.0184\\
BMStoGradMP & $1.05\times10^{-15}$ & 0.0023\\
M-FOCUSS & $3.80\times 10^{-3}$ & 0.0190\\
M-OMP & $7.45\times 10^{-16}$ & 0.0492\\
M-SP & $6.17\times 10^{-16}$ & 0.0217\\
M-CoSaMP & $8.22\times 10^{-16}$ & 0.0218\\
M-HTP & $5.33\times 10^{-16}$ & 0.0912 \\
\hline\hline
\end{tabular}
\vspace{4pt}
\caption{Comparison of the proposed algorithms and other MMV algorithms.}\label{tab:exp1}
\vspace{-4pt}
\end{table}

\subsection{Joint Sparse Video Sequence Recovery}
In this set of experiments, we compare the proposed Algorithm~\ref{alg:BMStoGradMP}, the split Bregman algorithm for constrained MMV problem (SBC) \cite[Algorithm~2]{Jian2015split}, M-FOCUSS, M-OMP, M-SP, M-CoSaMP and M-HTP, on joint sparse video sequence reconstruction. Note that SBC is based on $\ell_1$-minimization rather than $\ell_0$-norm constrained least-squares. We first download a candle video consisting of 75 frames from the Dynamic Texture Toolbox in \url{http://www.vision.jhu.edu/code/}. In order to make the test video sequence possess a joint sparse structure, we extract 11 frames of the original data, i.e., frames 1 to 7, 29, 37, 69 and 70, each of which is of size $80 \times 30$. Then we create a data matrix $X \in \mathbb{R}^{2400 \times 11}$, whose columns are a vectorization of all video frames. To further obtain a sparse representation of $X$, K-SVD \cite{aharon2006ksvd} is applied to obtain a dictionary $\Psi \in \mathbb{R}^{2400 \times 50}$ for $X$. The K-SVD dictionary $\Psi$ and the support of the corresponding coefficient matrix $\Theta$ for the extracted 11 frames are shown in Figure \ref{Dksvd} (a) and (b). Some selected columns of the dictionary $\Psi$, namely \emph{atoms}, are reshaped as an image of size $80\times 30$ illustrated in Figure \ref{Dksvd} (c) and (d).
It can be seen that these 11 frames are nearly joint sparse under the learned dictionary. The relative error of using this K-SVD dictionary $\Psi$ to represent the data matrix $X$ is $\frac{\|X-\Psi \Theta\|_F}{\|X\|_F}= 0.0870$. A Gaussian random matrix $\Phi \in \mathbb{R}^{60 \times 2400}$ with zero mean and unit variance is set as a sensing matrix, which is used to measure this data matrix. That is, the measurements $Y\in \mathbb{R}^{60\times 11}$ are generated via $Y=\Phi X$. Given the measurements $Y$ and the new sparse representation dictionary $A = \Phi\Psi$, we then apply Algorithm~\ref{alg:BMStoGradMP} and SBC to recover the joint sparse coefficient matrix $\hat{\Theta}$.
In Algorithm~\ref{alg:BMStoGradMP}, the sparsity level $k$ is set as $10$, and the block size $b$ is set as $3$, which implies that there are $d=20$ blocks. In addition, we set $\eta_1 = \eta_2 = 1$. Both StoGradMP and SBC stop when the residual error reaches a tolerance threshold $\tau = 10^{-6}$, i.e., $\|Y-A\hat{\Theta}\|_F \leq \tau$. In Table~\ref{tab:exp2}, we compare the reconstruction error, i.e., $\|Y-A\hat{\Theta}\|_F/\norm{Y}_F$, and running time for all methods being compared. One can see that the proposed Algorithm~\ref{alg:BMStoGradMP} is faster than all other comparing methods with relatively high accuracy.

\begin{figure}[h]
\centering
\begin{tabular}{cc}
\includegraphics[width=.5\textwidth]{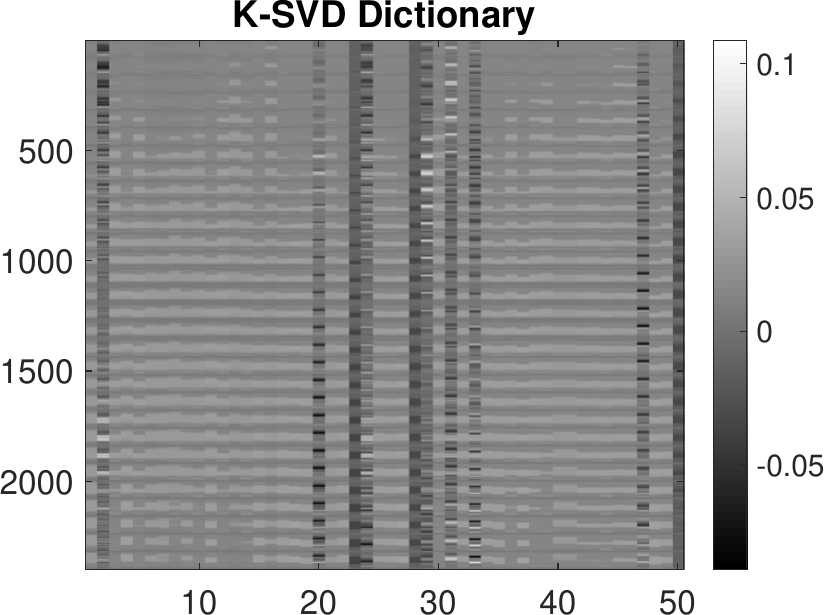}&
\includegraphics[width=.46\textwidth]{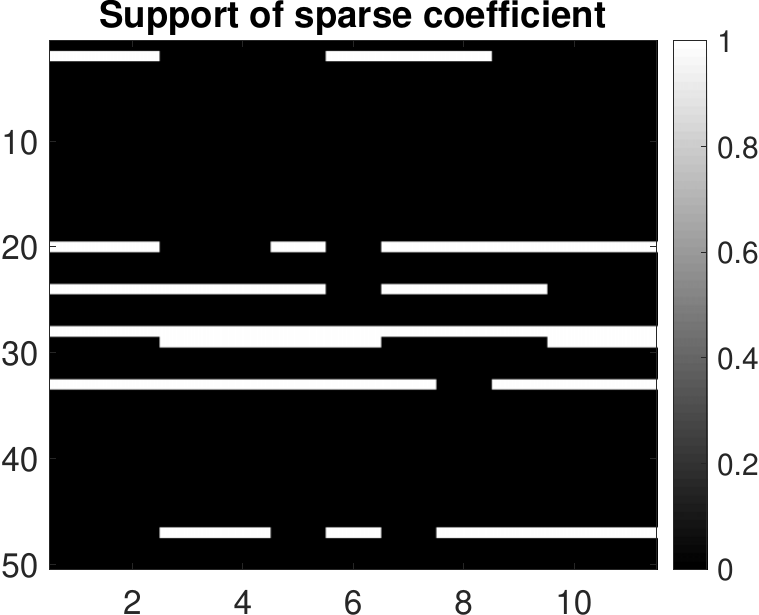}\\
(a)&(b)\\[4pt]
\includegraphics[width=.46\textwidth]{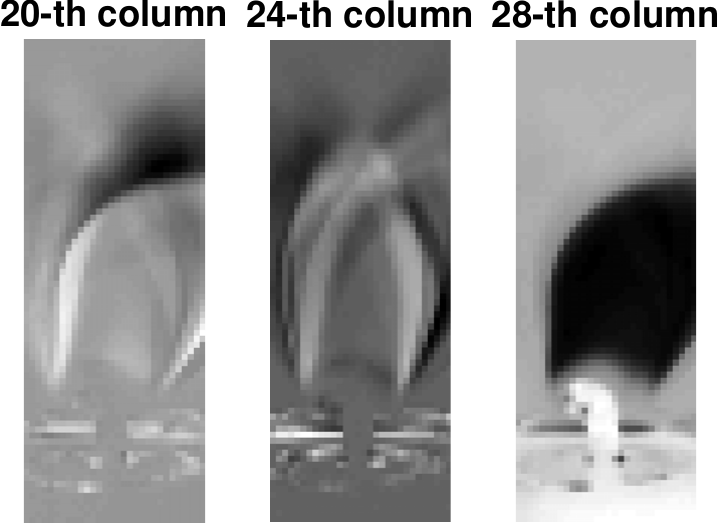}&
\includegraphics[width=.46\textwidth]{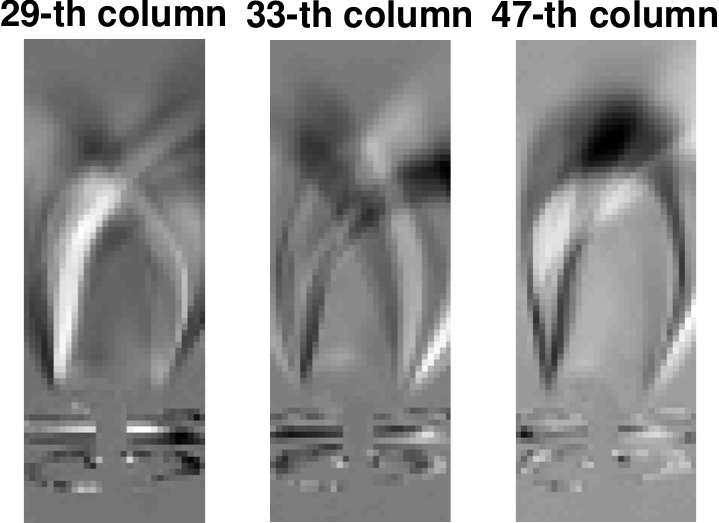}\\
(c)&(d)
\end{tabular}
\caption{(a)~K-SVD dictionary learned from the total 75 frames from the candle video. (b) Support of sparse coefficient matrix for extracted 11 frames. The sparse coefficient matrix has non-zero entries on white area. (c-d)~Some columns of the learned K-SVD dictionary.}
\label{Dksvd}
\end{figure}

\begin{table}
\centering
\begin{tabular}{p{1in}p{1in}c}
\hline\hline
Method & Error & Runtime (sec.) \\ \hline
BMStoGradMP & $3.09\times10^{-15}$ & $2.56\times10^{-4}$\\
M-FOCUSS & $3.21\times 10^{-3}$ & 0.2452\\
M-OMP & $2.51\times 10^{-16}$ & 0.0743\\
M-SP & $2.06\times 10^{-16}$ & 0.1297\\
M-CoSaMP & $4.07\times 10^{-16}$ & 0.0115\\
M-HTP & $5.72\times 10^{-2}$ & 0.0134 \\
SBC& $1.19\times 10^{-8}$ & $6.28\times10^{-4}$\\
\hline\hline
\end{tabular}
\vspace{4pt}
\caption{Performance comparison on joint sparse video recovery.}\label{tab:exp2}
\vspace{-4pt}
\end{table}

\section{Conclusions}\label{sec:con}
In this paper, we study the MMV joint sparse signal reconstruction problem, which is of great importance in a large amount of signal processing applications. One straightforward way is to columnwise concatenate results obtained by solving the SMV sparse signal recovery problem. However, concatenated SMV algorithms may not enforce joint sparsity of the approximated solutions. To address this issue, we propose two stochastic greedy algorithms, MStoIHT and MStoGradMP, together with their respective accelerated versions by applying the mini-batching technique. Our convergence analysis has shown that the proposed algorithms converge faster than their concatenated SMV counterparts. Moreover, we theoretically justify the applicability of the proposed algorithms to the distributed compressive sensing problem. To the best of our knowledge, this is the first work to use the term ``concatenated SMV algorithms'' with theoretical discussions on their convergence. A variety of numerical experiments on linear systems and video recovery have demonstrated that the proposed algorithms outperform the concatenated SMV algorithms in terms of efficiency, accuracy and robustness to the noise.

\section*{Appendix}
\renewcommand{\thesubsection}{\Alph{subsection}}
\subsection{Proof of Lemma~\ref{lem:rss}}
\begin{proof}
For any $X',X\in\mathbb{R}^{n\times L}$ with $|\supp_{\mathcal{D}}^r(X)\cup\supp_{\mathcal{D}}^r(X')|\leq k$, all columns of $X'$ and $X$ are $k$-sparse and $|\supp_{\mathcal{D}}(X_{\cdot,j})\cup\supp_{\mathcal{D}}^r(X'_{\cdot,j})|\leq k$. By \eqref{eqn:df} and $\mathcal{D}$-RSS of each $g_{i,j}$, we have for $i=1,\ldots,M$
\[\begin{aligned}
\norm{\frac{\partial }{\partial X}f_i(X)-\frac{\partial }{\partial X}f_i(X')}_F
&=\sqrt{\sum_{j=1}^L\norm{\nabla g_{i,j}(X_{\cdot,j})-\nabla g_{i,j}(X_{\cdot,j}')}_2^2}\\
&\leq \max\limits_{1\leq j\leq L}\rho_k^+(j)\norm{X-X'}_F:=\mu_k^+(i)\norm{X-X'}_F.
\end{aligned}\]
\end{proof}

\subsection{Proof of Lemma~\ref{lem:rsc}}
\begin{proof}
Similar to the previous lemma, for any $X',X\in\mathbb{R}^{n\times L}$ with $|\supp_{\mathcal{D}}^r(X)\cup\supp_{\mathcal{D}}^r(X')|\leq k$, we have
\[\begin{aligned}
&\phantom{=}F(X')-F(X)-\langle \frac{\partial }{\partial X}F(X),X'-X\rangle\\
&=\frac1L\sum_{j=1}^L\widehat{F_j}(X'_{\cdot,j})-\widehat{F_j}(X_{\cdot,j})-\frac1L\sum_{j=1}^L\langle \nabla \widehat{F_j}(X_{\cdot,j}),X'_{\cdot,j}-X_{\cdot,j}\rangle\\
&\geq \frac1L\sum_{j=1}^L\frac{\rho_{k,j}^{-}}{2}\norm{X_{\cdot,j}-X'_{\cdot,j}}_2^2\\
&\geq \frac{\mu_k^{-}}{2}\norm{X-X'}_F^2,
\end{aligned}\]
where $\mu_k^{-}=\min\limits_{1\leq j\leq L}\rho_{k,j}^-$.
\end{proof}

\section*{Acknowledgments}
The initial research for this effort was conducted at the Research Collaboration Workshop for Women in Data Science and
Mathematics, July 17-21 held at ICERM. Funding for the workshop was provided by ICERM, AWM and DIMACS (NSF grant CCF-1144502). Qin is supported by NSF DMS-1941197. Li was supported by the NSF grants CCF-1409258, CCF-1704204 and the DARPA Lagrange Program under ONR/SPAWAR contract N660011824020.
Needell is supported by NSF CAREER DMS-1348721 and NSF BIGDATA DMS-1740325.

\bibliographystyle{AIMS}
\bibliography{ref}

\providecommand{\href}[2]{#2}
\providecommand{\arxiv}[1]{\href{http://arxiv.org/abs/#1}{arXiv:#1}}
\providecommand{\url}[1]{\texttt{#1}}
\providecommand{\urlprefix}{URL }
\begin{thebibliography}{10}

\bibitem{agarwal2010fast}
\newblock A.~Agarwal, S.~Negahban and M.~J. Wainwright,
\newblock Fast global convergence rates of gradient methods for
  high-dimensional statistical recovery,
\newblock in \emph{Advances in Neural Information Processing Systems}, 2010,
\newblock 37--45.

\bibitem{aharon2006ksvd}
\newblock M.~Aharon, M.~Elad and A.~Bruckstein,
\newblock {K-SVD: An algorithm for designing overcomplete dictionaries for
  sparse representation},
\newblock \emph{IEEE Transactions on Signal Processing}, \textbf{54} (2006),
  4311--4322.

\bibitem{baron2006distributed}
\newblock D.~Baron, M.~B. Wakin, M.~F. Duarte, S.~Sarvotham and R.~G. Baraniuk,
\newblock Distributed compressed sensing,
\newblock \emph{IEEE Transactions on Information Theory}, \textbf{52} (2006),
  5406--5425.

\bibitem{bazerque2010distributed}
\newblock J.~A. Bazerque and G.~B. Giannakis,
\newblock Distributed spectrum sensing for cognitive radio networks by
  exploiting sparsity,
\newblock \emph{IEEE Transactions on Signal Processing}, \textbf{58} (2010),
  1847--1862.

\bibitem{blanchard2014greedy}
\newblock J.~D. Blanchard, M.~Cermak, D.~Hanle and Y.~Jing,
\newblock Greedy algorithms for joint sparse recovery,
\newblock \emph{IEEE Transactions on Signal Processing}, \textbf{62} (2014),
  1694--1704.

\bibitem{blanchard2015cgiht}
\newblock J.~D. Blanchard, J.~Tanner and K.~Wei,
\newblock {CGIHT}: conjugate gradient iterative hard thresholding for
  compressed sensing and matrix completion,
\newblock \emph{Information and Inference: A Journal of the IMA}, \textbf{4}
  (2015), 289--327.

\bibitem{blumensath2009iterative}
\newblock T.~Blumensath and M.~E. Davies,
\newblock Iterative hard thresholding for compressed sensing,
\newblock \emph{Applied and Computational Harmonic Analysis}, \textbf{27}
  (2009), 265--274.

\bibitem{blumensath2010normalized}
\newblock T.~Blumensath and M.~E. Davies,
\newblock Normalized iterative hard thresholding: Guaranteed stability and
  performance,
\newblock \emph{IEEE Journal of selected topics in signal processing},
  \textbf{4} (2010), 298--309.

\bibitem{burden2004numerical}
\newblock R.~Burden and J.~Faires,
\newblock \emph{Numerical analysis},
\newblock Cengage Learning, 2004.

\bibitem{chen2006theoretical}
\newblock J.~Chen and X.~Huo,
\newblock Theoretical results on sparse representations of multiple-measurement
  vectors,
\newblock \emph{IEEE Transactions on Signal Processing}, \textbf{54} (2006),
  4634--4643.

\bibitem{cotter2005sparse}
\newblock S.~F. Cotter, B.~D. Rao, K.~Engan and K.~Kreutz-Delgado,
\newblock Sparse solutions to linear inverse problems with multiple measurement
  vectors,
\newblock \emph{IEEE Transactions on Signal Processing}, \textbf{53} (2005),
  2477--2488.

\bibitem{dai2009subspace}
\newblock W.~Dai and O.~Milenkovic,
\newblock Subspace pursuit for compressive sensing signal reconstruction,
\newblock \emph{IEEE Transactions on Information Theory}, \textbf{55} (2009),
  2230--2249.

\bibitem{davies2012rank}
\newblock M.~E. Davies and Y.~C. Eldar,
\newblock Rank awareness in joint sparse recovery,
\newblock \emph{IEEE Transactions on Information Theory}, \textbf{58} (2012),
  1135--1146.

\bibitem{donoho1993unconditional}
\newblock D.~L. Donoho,
\newblock Unconditional bases are optimal bases for data compression and for
  statistical estimation,
\newblock \emph{Applied and computational harmonic analysis}, \textbf{1}
  (1993), 100--115.

\bibitem{durgin2019fast}
\newblock N.~Durgin, C.~Huang, R.~Grotheer, S.~Li, A.~Ma, D.~Needell and
  J.~Qin,
\newblock Fast hyperspectral diffuse optical imaging method with joint
  sparsity,
\newblock in \emph{41st Annual International Conference of the IEEE Engineering
  in Medicine \& Biology Society (EMBC)}, 2019.

\bibitem{Elad2010}
\newblock M.~Elad,
\newblock \emph{{Sparse and Redundant Representations From Theory to
  Applications in Signal and Image Processing}},
\newblock Springer, 2010.

\bibitem{feng1996spectrum}
\newblock P.~Feng and Y.~Bresler,
\newblock Spectrum-blind minimum-rate sampling and reconstruction of multiband
  signals,
\newblock in \emph{Acoustics, Speech and Signal Processing (ICASSP), 1996 IEEE
  International Conference on}, vol.~3,
\newblock IEEE, 1996,
\newblock 1688--1691.

\bibitem{foucart2011hard}
\newblock S.~Foucart,
\newblock Hard thresholding pursuit: an algorithm for compressive sensing,
\newblock \emph{SIAM Journal on Numerical Analysis}, \textbf{49} (2011),
  2543--2563.

\bibitem{foucart2013mathematical}
\newblock S.~Foucart and H.~Rauhut,
\newblock \emph{A mathematical introduction to compressive sensing}, vol.~1,
\newblock Birkh{\"a}user Basel, 2013.

\bibitem{giryes2014greedy}
\newblock R.~Giryes, S.~Nam, M.~Elad, R.~Gribonval and M.~E. Davies,
\newblock Greedy-like algorithms for the cosparse analysis model,
\newblock \emph{Linear Algebra and its Applications}, \textbf{441} (2014),
  22--60.

\bibitem{he2008cg}
\newblock Z.~He, A.~Cichocki, R.~Zdunek and J.~Cao,
\newblock {CG-M-FOCUSS and its application to distributed compressed sensing},
\newblock in \emph{International Symposium on Neural Networks},
\newblock Springer, 2008,
\newblock 237--245.

\bibitem{Jian2015split}
\newblock Z.~Jian, F.~Yuli, Z.~Qiheng and L.~Haifeng,
\newblock Split bregman algorithms for multiple measurement vector problem,
\newblock \emph{Multidim Syst Sign Process}.

\bibitem{kim2012compressive}
\newblock J.~M. Kim, O.~K. Lee and J.~C. Ye,
\newblock {Compressive MUSIC: Revisiting the link between compressive sensing
  and array signal processing},
\newblock \emph{IEEE Transactions on Information Theory}, \textbf{58} (2012),
  278--301.

\bibitem{lee2012subspace}
\newblock K.~Lee, Y.~Bresler and M.~Junge,
\newblock Subspace methods for joint sparse recovery,
\newblock \emph{IEEE Transactions on Information Theory}, \textbf{58} (2012),
  3613--3641.

\bibitem{li2017atomic}
\newblock S.~Li, D.~Yang, G.~Tang and M.~B. Wakin,
\newblock Atomic norm minimization for modal analysis from random and
  compressed samples,
\newblock \emph{IEEE Transactions on Signal Processing}, \textbf{66} (2018),
  1817--1831.

\bibitem{lu2011fast}
\newblock H.~Lu, X.~Long and J.~Lv,
\newblock A fast algorithm for recovery of jointly sparse vectors based on the
  alternating direction methods,
\newblock in \emph{Proceedings of the Fourteenth International Conference on
  Artificial Intelligence and Statistics}, 2011,
\newblock 461--469.

\bibitem{magnus2010concept}
\newblock J.~R. Magnus,
\newblock On the concept of matrix derivative,
\newblock \emph{Journal of Multivariate Analysis}, \textbf{101} (2010),
  2200--2206.

\bibitem{majumdar2013rank}
\newblock A.~Majumdar and R.~Ward,
\newblock Rank awareness in group-sparse recovery of multi-echo {MR} images,
\newblock \emph{Sensors}, \textbf{13} (2013), 3902--3921.

\bibitem{majumdar2011joint}
\newblock A.~Majumdar and R.~K. Ward,
\newblock Joint reconstruction of multiecho mr images using correlated
  sparsity,
\newblock \emph{Magnetic resonance imaging}, \textbf{29} (2011), 899--906.

\bibitem{majumdar2012face}
\newblock A.~Majumdar and R.~K. Ward,
\newblock {Face recognition from video: An MMV recovery approach},
\newblock in \emph{Acoustics, Speech and Signal Processing (ICASSP), 2012 IEEE
  International Conference on},
\newblock IEEE, 2012,
\newblock 2221--2224.

\bibitem{mishali2008reduce}
\newblock M.~Mishali and Y.~C. Eldar,
\newblock Reduce and boost: Recovering arbitrary sets of jointly sparse
  vectors,
\newblock \emph{IEEE Transactions on Signal Processing}, \textbf{56} (2008),
  4692--4702.

\bibitem{needell2009cosamp}
\newblock D.~Needell and J.~A. Tropp,
\newblock {CoSaMP: Iterative signal recovery from incomplete and inaccurate
  samples},
\newblock \emph{Applied and Computational Harmonic Analysis}, \textbf{26}
  (2009), 301--321.

\bibitem{needell2010signal}
\newblock D.~Needell and R.~Vershynin,
\newblock Signal recovery from incomplete and inaccurate measurements via
  regularized orthogonal matching pursuit,
\newblock \emph{IEEE Journal of Selected Topics in Signal Processing},
  \textbf{4} (2010), 310--316.

\bibitem{needell2016batched}
\newblock D.~Needell and R.~Ward,
\newblock {Batched Stochastic Gradient Descent with Weighted Sampling},
\newblock in \emph{International Conference Approximation Theory},
\newblock Springer, 2016,
\newblock 279--306.

\bibitem{nguyen2014linear}
\newblock N.~Nguyen, D.~Needell and T.~Woolf,
\newblock Linear convergence of stochastic iterative greedy algorithms with
  sparse constraints,
\newblock \emph{IEEE Transactions on Information Theory}.

\bibitem{nguyenunified}
\newblock N.~Nguyen, S.~Chin and T.~Tran,
\newblock {A Unified Iterative Greedy Algorithm for Sparsity-Constrained
  Optimization}, 2012.

\bibitem{pati1993orthogonal}
\newblock Y.~C. Pati, R.~Rezaiifar and P.~S. Krishnaprasad,
\newblock Orthogonal matching pursuit: Recursive function approximation with
  applications to wavelet decomposition,
\newblock in \emph{Proceedings of 27th Asilomar conference on signals, systems
  and computers},
\newblock IEEE, 1993,
\newblock 40--44.

\bibitem{schmidt1986multiple}
\newblock R.~Schmidt,
\newblock Multiple emitter location and signal parameter estimation,
\newblock \emph{IEEE transactions on antennas and propagation}, \textbf{34}
  (1986), 276--280.

\bibitem{tropp2004greed}
\newblock J.~A. Tropp,
\newblock {Greed is good: Algorithmic results for sparse approximation},
\newblock \emph{IEEE Transactions on Information theory}, \textbf{50} (2004),
  2231--2242.

\bibitem{tropp2006algorithms}
\newblock J.~A. Tropp, A.~C. Gilbert and M.~J. Strauss,
\newblock {Algorithms for simultaneous sparse approximation. Part I: Greedy
  pursuit},
\newblock \emph{Signal Processing}, \textbf{86} (2006), 572--588.

\bibitem{yuan2014gradient}
\newblock X.~Yuan, P.~Li and T.~Zhang,
\newblock Gradient hard thresholding pursuit for sparsity-constrained
  optimization,
\newblock in \emph{International Conference on Machine Learning}, 2014,
\newblock 127--135.

\bibitem{zhang2011sparse}
\newblock T.~Zhang,
\newblock Sparse recovery with orthogonal matching pursuit under {RIP},
\newblock \emph{IEEE Transactions on Information Theory}, \textbf{57} (2011),
  6215--6221.

\end{thebibliography}

\medskip
Received xxxx 20xx; revised xxxx 20xx.
\medskip
\end{document}